\documentclass[envcountsame]{llncs}

\newif\ifignore 
\ignorefalse
\newcommand{\auxproof}[1]{
\ifignore\mbox{}\newline
\textbf{PROOF:} \dotfill\newline
{\it #1}\mbox{}\newline
\textbf{ENDPROOF}\dotfill
\fi}

\usepackage{etex} 
\usepackage{amsmath}
\usepackage{amssymb}
\usepackage{amstext}
\usepackage{mathtools}
\usepackage{mathrsfs}
\usepackage{stmaryrd}
\usepackage{xspace}
\usepackage{enumitem}
\usepackage{cmll}
\usepackage{hyperref}
\usepackage{url}
\usepackage{nicefrac}
\usepackage{fancybox}
\usepackage{listings}
\usepackage{wrapfig}
\usepackage{xypic}
\usepackage[all,cmtip]{xy}
\xyoption{2cell}
\UseTwocells
\xyoption{tips}

\usepackage{tikz}
\usetikzlibrary{circuits.ee.IEC}
\usetikzlibrary{shapes,shapes.geometric,shapes.misc}

\ifdefined\mathsl\else
  \DeclareMathAlphabet{\mathsl}{\encodingdefault}{\rmdefault}{\mddefault}{\sldefault}
  \SetMathAlphabet{\mathsl}{bold}{\encodingdefault}{\rmdefault}{\bfdefault}{\sldefault}
\fi

\newcommand{\QEDbox}{\square}
\newcommand{\QED}{\hspace*{\fill}$\QEDbox$}

\makeatletter
\newcommand{\mathoverlap}[2]{\mathpalette\mathoverlap@{{#1}{#2}}}
\newcommand{\mathoverlap@}[2]{\mathoverlap@@{#1}#2}
\newcommand{\mathoverlap@@}[3]{\ooalign{$\m@th#1#2$\crcr\hidewidth$\m@th#1#3$\hidewidth}}
\makeatother

\newcommand{\klafter}{\mathbin{\mathoverlap{\circ}{\cdot}}}

\DeclareSymbolFont{T1op}{T1}{cmr}{m}{n}
\SetSymbolFont{T1op}{bold}{T1}{cmr}{bx}{n}
\DeclareMathSymbol{\mathguilsinglleft}{\mathopen}{T1op}{'016}
\DeclareMathSymbol{\mathguilsinglright}{\mathclose}{T1op}{'017}

\newcommand{\idmap}[1][]{\ensuremath{\mathrm{id}_{#1}}}
\newcommand{\after}{\mathrel{\circ}}

\newcommand{\binomial}[1][]{\ensuremath{\mathsl{binom}[#1]}}
\newcommand{\multinomial}[1][]{\ensuremath{\mathsl{mulnom}[#1]}}
\newcommand{\Beta}{\ensuremath{\mathsl{B}}}
\newcommand{\betachan}{\ensuremath{\mathsl{beta}}}
\newcommand{\betabinomial}[1][]{\ensuremath{\mathsl{betabin}[#1]}}
\newcommand{\draw}{\ensuremath{\mathsl{D}}}
\newcommand{\drawdelete}{\ensuremath{\mathsl{DD}}}

\newcommand{\setin}[3]{\{#1\in#2\;|\;#3\}}
\newcommand{\supp}{\mathrm{supp}}

\newcommand{\ket}[1]{\ensuremath{|{\kern.1em}#1{\kern.1em}\rangle}}
\newcommand{\bigket}[1]{\ensuremath{\big|{\kern.1em}#1{\kern.1em}\big\rangle}}

\newcommand{\zero}{\ensuremath{\mathbf{0}}}

\newcommand{\intd}{{\kern.2em}\mathrm{d}{\kern.03em}}
\newcommand{\indic}[1]{\mathbf{1}_{#1}}

\newcommand{\dd}{\mathrm{d}}
\newcommand{\iid}{\mathsl{iid}}
\newcommand{\coin}{\mathsl{bern}}
\newcommand{\dnib}{\mathrel{\scalebox{-0.5}[1]{$\gg=$}}}

\newcommand{\distributionsymbol}{\mathcal{D}}

\newcommand{\Dst}{\distributionsymbol}
\newcommand{\multisetsymbol}{\mathcal{M}}

\newcommand{\Mlt}{\multisetsymbol}
\newcommand{\neMlt}{\multisetsymbol_{*}}
\newcommand{\Giry}{\mathcal{G}}

\newcommand{\Cat}[1]{\ensuremath{\mathbf{#1}}\xspace}

\newcommand{\Kl}{\mathcal{K}{\kern-.4ex}\ell}
\newcommand{\EM}{\mathcal{E}{\kern-.4ex}\mathcal{M}}
\newcommand{\Sets}{\Cat{Sets}}

\newcommand{\NNO}{\mathbb{N}}

\newcommand{\R}{\mathbb{R}}
\newcommand{\nnR}{\R_{\geq 0}}
\newcommand{\pR}{\R_{> 0}}

\newsavebox\sbpto
\savebox\sbpto{\begin{tikzpicture}[baseline=-2.5pt]
            \filldraw[fill=white,draw=white] circle (1.4pt);
            \filldraw[fill=white,draw=black,line width=0.2pt] circle (1.2pt);
                \end{tikzpicture}}
\newcommand\chanto{\mathrel{\ooalign{$\rightarrow$\cr
            \hfil\!$\usebox\sbpto$\hfil\cr}}}

\makeatother
 \newdir{ >}{{}*!/-7.5pt/@{>}}
 \newdir{|>}{!/4.5pt/@{|}*:(1,-.2)@^{>}*:(1,+.2)@_{>}}
 \newdir{ |>}{{}*!/-3pt/@{|}*!/-7.5pt/:(1,-.2)@^{>}*!/-7.5pt/:(1,+.2)@_{>}}

\newsavebox\sbground
\savebox\sbground{\begin{tikzpicture}[circuit ee IEC,yscale=0.5,xscale=0.4]
                \draw (0,-2ex) to (0,0) node[ground,rotate=90,xshift=.65ex] {};
                \end{tikzpicture}}

\newcommand{\ie}{\textit{i.e.}\xspace}
\newcommand{\eg}{\textit{e.g.}\xspace}

\DeclareFixedFont{\ttb}{T1}{txtt}{bx}{n}{11} 
\DeclareFixedFont{\ttm}{T1}{txtt}{m}{n}{11}  
\usepackage{color}
\definecolor{deepblue}{rgb}{0,0,0.5}
\definecolor{deepred}{rgb}{0.6,0,0}
\definecolor{deepgreen}{rgb}{0,0.5,0}
\definecolor{lightgray}{rgb}{0.83,0.83,0.83}

\newcommand\pythonstyle{\lstset{
backgroundcolor = \color{lightgray},
language=Python,
basicstyle=\ttm,
otherkeywords={self,>>>,...},             
keywordstyle=\ttb\color{deepblue},
emph={MyClass,__init__},          
emphstyle=\ttb\color{deepred},    
stringstyle=\color{deepgreen},
frame=tb,                         
showstringspaces=false            %
}}

\lstnewenvironment{python}[1][]
{
\pythonstyle
\lstset{#1}
}
{}

\newcommand\pythoninline[1]{{\pythonstyle\lstinline!#1!}}

\begin{document}
\begin{frontmatter}
  \title{De Finetti's Construction as a Categorical Limit
  }


\author{Bart Jacobs\inst{1} \and Sam Staton\inst{2}}

\institute{Institute for Computing and Information Sciences,
\\ 
Radboud University, Nijmegen, NL
\\ 
\email{bart@cs.ru.nl}
\and
Department of Computer Science, 
\\
University of Oxford, UK
\\
\email{sam.staton@cs.ox.ac.uk}
}

\maketitle

\begin{abstract}
This paper reformulates a classical result in probability theory from
the 1930s in modern categorical terms: de Finetti's representation
theorem is redescribed as limit statement for a chain of finite spaces
in the Kleisli category of the Giry monad.
This new limit is used to
identify among exchangeable coalgebras the final one.
\end{abstract}

\end{frontmatter}


\renewcommand{\thepage}{\arabic{page}}

\section{Introduction}\label{IntroSec}

An indifferent belief about the bias of a coin can be expressed in
various ways.  One way is to use elementary sentences like
\begin{equation}
  \begin{array}{l}
    \textit{``When we toss the coin $7$ times, the probability of $5$ heads is $\tfrac 1 8$.''}
   \\\textit{``When we toss the coin $k$ times, the probability of $n$ heads is $\tfrac 1 {k+1}$.''}
    \end{array}
  \label{eqn:textb}
\end{equation}
Another way is to give a probability measure on the space~%
$[0,1]$ of biases:
\begin{equation}
  \textit{``The bias of the coin is uniformly distributed across the interval $[0,1]$.''}
  \label{eqn:texta}
\end{equation}

\noindent This sentence~\eqref{eqn:texta} is more conceptually
challenging than \eqref{eqn:textb}, but we may
deduce~\eqref{eqn:textb} from~\eqref{eqn:texta} using Lebesgue
integration:
\[ 
\displaystyle\int_0^1 \textstyle \binom 7 5r^5(1-r)^2 \ \dd r 
=
\tfrac 1 {8},
\qquad\qquad
\displaystyle\int_0^1 \textstyle \binom k nr^n(1-r)^{k-n} \ \dd r 
=
\tfrac 1 {k+1}
\]
where $\binom k n=\frac {k!}{n!(k-n)!}$ is the binomial coefficient.
Here we have considered indifferent beliefs as a first illustration,
but more generally de Finetti's theorem gives a passage from sentences
about beliefs like~\eqref{eqn:textb} to sentences
like~\eqref{eqn:texta}.

\begin{theorem}[De Finetti, Informal]
If we assign a probability to all finite experiment outcomes involving
a coin (such as~\eqref{eqn:textb}), and these probabilities are all
suitably consistent, then by doing so we have given a probability measure on $[0,1]$
(such as the uniform distribution, 
\eqref{eqn:texta}). 
\end{theorem}

In this paper we give two categorical / coalgebraic formulations of
this theorem.
\begin{itemize}
\item In Theorem~\ref{FinettiThm}, we state the result as: \emph{the
  unit interval $[0,1]$ is an inverse limit over a chain of finite
  spaces in a category of probability channels}. In particular, a cone
  over this diagram is an assignment of finite probabilities such
  as~\eqref{eqn:textb} that is suitably consistent.

\item In Theorem~\ref{FinalThm}, we state the result as: \emph{$[0,1]$
  carries a final object in a category of (exchangeable) coalgebras
  for the functor $2\times (-)$ in the category of probability
  channels.} A coalgebra can be thought of as a process that
  coinductively generates probabilities of the
  form~\eqref{eqn:textb}. A coalgebra is a little more data than
  strictly necessary, but this extra data corresponds to the concept
  of sufficient statistics which is useful in Bayesian inference.
\end{itemize}

\paragraph{Multisets and probability distributions.}
A multiset is a `set' in which elements may occur multiple times. We
can write such a multiset over a set $\{R,G,B\}$, representing red,
green and blue balls, as:
\[ 3\ket{R} + 5\ket{G} + 2\ket{B}. \]

\noindent This multiset can be seen as an urn containing three red,
five green and two blue balls.

A probability distribution is a convex combination of elements, where the 
frequencies (or probabilities) with which elements occur are numbers
in the unit interval $[0,1]$ that add up to one, as in:
\[ \textstyle\frac{3}{10}\ket{R} + \frac{1}{2}\ket{G} + \frac{1}{5}\ket{B}. \]

Obviously, distributions, whether discrete (as above) or continuous,
play a central role in probability theory.  Multisets also play a
crucial role, for instance, an urn with coloured balls is very
naturally represented as a multiset. Similarly, multiple data items in
learning scenarios are naturally captured by multisets (see
\eg~\cite{Jacobs19d}).  The two concepts are related, since any
multiset gives rise to a probability distribution, informally by
thinking of it as an urn and then drawing from it (sampling with
replacement).

This paper develops another example of how the interplay between
multisets and distributions involves interesting structure, in an
emerging area of categorical probability theory.

\paragraph{Beliefs as cones.}
In brief, following an experiment involving $K$ tosses of a coin,
we can set up an urn that approximately simulates the coin according to the frequencies that we have seen.
So if we see $5$ heads out of $7$ tosses, we put $5$ black balls into an urn together with $2$ white balls, so that the chance of black(=heads) is $\frac 57$.
We can describe a belief about a coin by describing the probabilities of ending up with
certain urns (multisets) after certain experiments. Clearly, if we draw and discard a ball from an urn,
then this is the same as stopping the experiment one toss sooner. 
So a consistent belief about our coin determines a cone over the chain 
\begin{equation}\label{eqn:chainintro}
  \xymatrix{\Mlt[1](2)&\ar[l]|-\circ\Mlt[2](2)&\ar[l]|-\circ \Mlt[3](2)&
    \ar[l]|-\circ\dots}
\end{equation}

\noindent in the category of finite sets and channels (aka probability
kernels). Here $\Mlt[K](2)$ is the set of multisets over $2$ of size
$K$, \ie~urns containing $K$ balls coloured black and white, and the
leftward arrows describe randomly drawing and discarding a ball. We
use the notation $\chanto$ for channels, a.k.a.\ Kleisli map, see
below for details.  Our formulation of de Finetti's theorem says that
the limit of this chain is~$[0,1]$. So to give a belief, \ie~a cone
over~\eqref{eqn:chainintro}, is to give a probability measure
on~$[0,1]$.  We set up the theory of multisets and cones in
Section~\ref{DrawSubsec}, proving the theorem in
Section~\ref{FinettiSec}.

\paragraph{Exchangeable coalgebras.}
In many situations, it is helpful to describe a belief about a coin in
terms of a stateful generative process, \ie~a channel (aka probability
kernel) $h\colon X\chanto 2\times X$, for some space~$X$, which gives
coinductively a probability for a sequence of outcomes.  This is a
coalgebra for the functor $2\times (-)$.  An important example is
P\'olya's urn, which is an urn-based process, but one where the urn
changes over time.  Another important example puts the carrier set
$X=[0,1]$ as the set of possible biases, and the process merely tosses
a coin with the given bias.  We show that $[0,1]$ carries the final
object among the `exchangeable' coalgebras.  Thus there is a canonical
channel $X\chanto [0,1]$, taking the belief described by a generative
process such as P\'olya's urn to a probability measure on $[0,1]$ such
as the uniform distribution.  We set up the theory of exchangeable
coalgebras in Section~\ref{PolyaSec}, proving the final coalgebra
theorem in Section~\ref{ExchangeSec} by using the limit reformulation
of de Finetti's theorem.

\section{Background on mulitsets and distributions}\label{BackgroundSec}
\paragraph{Multisets.}
A \emph{multiset} (or \emph{bag}) is a `set' in which elements may occur multiple
times.  We shall write $\Mlt(X)$ for the set of (finite) multisets with
elements from a set $X$. Elements $\varphi\in\Mlt(X)$ may be described
as functions $\varphi\colon X \rightarrow \NNO$ with finite support,
that is, with $\supp(\varphi) \coloneqq \setin{x}{X}{\varphi(x) \neq
  0}$ is finite.  Alternatively, multisets can be described as formal
sums $\sum_{i} n_{i}\ket{x_i}$, where $n_{i}\in\NNO$ describes the
multiplicity of $x_{i}\in X$, that is, the number of times that the
element $x_i$ occurs in the multiset. The mapping $X \mapsto \Mlt(X)$
is a monad on the category $\Sets$ of sets and functions. In fact,
$\Mlt(X)$ with pointwise addition and empty multiset $\zero$, is the
free commutative monoid on $X$.

Frequently we need to have a grip on the total number of elements
occurring in a multiset. Therefore we write, for $K\in\NNO$,
\[ \begin{array}{rcl}
\Mlt[K](X)
& \coloneqq &
\setin{\varphi}{\Mlt(X)}{\mbox{$\varphi$ has precisely $K$ elements}}
\\
& = &
\setin{\varphi}{\Mlt(X)}{\sum_{x\in X}\varphi(x) = K}.
\end{array} \]

\noindent Clearly, $\Mlt[0](X)$ is a singleton, containing only the
empty multiset $\zero$.  Each $\Mlt[K](X)$ is a quotient of the set of
sequences $X^K$ via an accumulator map:
\begin{equation}
\label{AccDiag}
\vcenter{\xymatrix{
X^{K} \ar[r]^-{\mathsl{acc_X}} & \Mlt[K](X)
\quad\mbox{with}\quad
\mathsl{acc}_X(x_{1}, \ldots, x_{K})
\coloneqq
\displaystyle\sum_{x\in X}\big(\textstyle \sum_{i} \indic{x_i}(x)\big)\bigket{x}
}}
\end{equation}

\noindent where $\indic{x_i}$ is the indicator function sending $x\in
X$ to $1$ if $x=x_{i}$ and to $0$ otherwise.  In particular, for
$X=2$, $\mathsl{acc}_2(b_1,\dots,b_K)={(\sum_{i}b_{i})\ket{1} +
  (K-\sum_{i}b_{i})\ket{0}}$.  The map $\mathsl{acc}_X$ is
permutation-invariant, in the sense that:
\[ \begin{array}{rcl}
\mathsl{acc}_X(x_{1}, \ldots, x_{K})
& = &
\mathsl{acc}_X(x_{\sigma(1)}, \ldots, x_{\sigma(K)})
\qquad\mbox{for each permutation }\sigma\in S_K.
\end{array} \]

\noindent Such permutation-invariance, or exchangeability, plays an important role in de Finetti's
work~\cite{Finetti30}, see the end of Section~\ref{ExchangeSec}.
 
\paragraph{Discrete probability.}
A \emph{discrete probability distribution} on a set $X$ is a finite
formal sum $\sum_{i}r_{i}\ket{x_i}$ of elements $x_{i}\in X$ with
multiplicity $r_{i}\in[0,1]$ such that $\sum_{i}r_{i} = 1$. It can
also be described as a function $X \rightarrow [0,1]$ with finite
support. We write $\Dst(X)$ for the set of distributions on $X$. This
gives a monad $\Dst$ on $\Sets$; it satisfies $\Dst(1) \cong 1$ and
$\Dst(2) \cong [0,1]$, where $1 = \{0\}$ and $2 = \{0,1\}$.

Maps in the associated Kleisli category $\Kl(\Dst)$ will be called
\emph{channels} and are written with a special arrow $\chanto$. Thus
$f\colon X \chanto Y$ denotes a function ${f\colon X \rightarrow
  \Dst(Y)}$. For a distribution $\omega\in\Dst(X)$ and a channel
$g\colon Y \chanto Z$ we define a new distribution $f \dnib \omega \in
\Dst(Y)$ and a new channel $g \klafter f \colon X \chanto Z$ by
Kleisli composition:
\[ \begin{array}{rclcrcl}
f \dnib \omega
& \coloneqq &
\displaystyle\sum_{y\in Y} \textstyle
   \Big(\sum_{x} \omega(x) \cdot f(x)(y)\Big)\bigket{y}
& \qquad &
\big(g \klafter f\big)(x)
& \coloneqq &
g \dnib f(x).
\end{array} \]

\paragraph{Discrete probability over multisets.}
For each number $K\in\NNO$ and probability
  $r\in [0,1]$ there is the familiar \emph{binomial} distribution
  $\binomial[K](r) \in \Dst(\{0,1,\ldots, K\})$. It captures
  probabilities for iterated flips of a known coin, and is given by the convex
  sum:
\[ \begin{array}{rcl}
\binomial[K](r)
& \coloneqq &
\displaystyle\sum_{0\leq k\leq K} \textstyle
   \binom{K}{k} \cdot r^{k} \cdot (1-r)^{K-k}\, \bigket{k}.
\end{array} \]

\noindent The multiplicity probability before $\ket{k}$ in this
expression is the chance of getting $k$ heads of out $K$ coin flips,
where each flip has fixed bias $r\in[0,1]$. In this way we obtain a channel
$\binomial[K] \colon [0,1] \chanto \{0,1,\ldots, K\}$.

More generally, one can define a multinomial channel:
\[ \xymatrix@C+1pc{
\Dst(X)\ar[rr]|-{\circ}^-{\multinomial[K]} & & \Mlt[K](X).
} \]

\noindent It is defined as:
\[ \begin{array}{rcl}
\multinomial[K](\omega)
& \coloneqq &
\displaystyle \sum_{\varphi\in\Mlt[K](X)}
   \frac{K!}{\prod_{x}\varphi(x)!}\cdot \textstyle
   {\displaystyle\prod}_{x} \omega(x)^{\varphi(x)} \bigket{\varphi}.
\end{array} \]

\noindent The binomial channel is a special case, since
$[0,1]\cong \Dst(2)$ and 
$\{0,1,\ldots,K\} \cong \Mlt[K](2)$, via $k \mapsto k\ket{0} +
(K-k)\ket{1}$. The binary case will play an important role in the
sequel. 


\section{Drawing from an urn}\label{DrawSubsec}
Recall our motivation. We have an unknown coin, and we are willing to
give probabilities to the outcomes of different finite experiments
with it. From this we ultimately want to say something about the coin.
The only experiments we consider comprise tossing a coin a fixed
number~$K$ times and recording the number of heads.  The coin has no
memory, and so the order of heads and tails doesn't matter.  So we can
imagine an urn containing black balls and white balls, and in an
experiment we start with an empty urn and put a black ball in the urn
for each head and a white ball for each tail.  The urn forgets the
order in which the heads and tails appear.  We can then describe our
belief about the coin in terms of our beliefs about the probability of
certain urns occurring in the different experiments. These
probabilities should be consistent in the following informal sense,
that connects the different experiments.
\begin{itemize}
  \item If we run the experiment with $K+1$ tosses, and then discard a random ball from the urn, it should be the same as running the experiment with just $K$ tosses. 
\end{itemize}
In more detail: because the coin tosses are exchangeable, that is, the
coin is memoryless, to run the experiment with $K+1$ tosses, is to run
the experiment with $K+1$ tosses and then report some chosen
permutation of the results. This justifies representing the outcome of
the experiment with an unordered urn. It follows that the statistics
of the experiment are the same if we choose a permutation of the
results at random. To stop the experiment after~$K$ tosses, i.e.~to
discard the ball corresponding to the last toss, is thus to discard a
ball unifomly at random.

As already mentioned: an urn filled with certain objects of the kind
$x_{1}, \ldots, x_{n}$ is very naturally described as a multiset
$\sum_{i}n_{i}\ket{x_i}$, where $n_{i}$ is the number of objects $x_i$
in the urn.  Thus an urn containing objects from a set $X$ is a
multiset in $\Mlt(X)$.  For instance, an urn with three black and two
white balls can be described as a multiset $3\ket{B} + 2\ket{W}\in
\Mlt(2)$, where $B$ stands for black and $W$ for white.  We can thus
formalise a belief about the $K$-th experiment, where we toss our coin
$K$ times, as a distribution $\omega_K\in \Dst(\Mlt[K](2))$, on urns
with $K$ black and white balls.

For the consistency property, we need to talk about drawing and
discarding from an urn.  We define a channel that draws a single item
from an urn:
\begin{equation}
\label{DrawDiag}
\xymatrix{
\Mlt[K+1](X)\ar[r]^-{\draw} & \Dst\Big(X\times\Mlt[K](X)\Big)
}
\end{equation}

\noindent It is defined as:
\[ \begin{array}{rcl}
\draw(\varphi)
& \coloneqq &
\displaystyle\sum_{x\in\supp(\varphi)}
\frac{\varphi(x)}{\sum_{y}\varphi(y)}\bigket{x, \varphi-1\ket{x}}.
\end{array} \]

\noindent Concretely, in the above example:
\[ \begin{array}{rcl}
\draw\big(3\ket{B} + 2\ket{W}\big)
& = &
\frac{3}{5}\bigket{B, 2\ket{B} + 2\ket{W}} + 
  \frac{2}{5}\bigket{W, 3\ket{B} + 1\ket{W}}.
\end{array} \]

\noindent Notice that on the right-hand-side we use brackets $\ket{-}$
inside brackets $\ket{-}$. The outer $\ket{-}$ are for $\Dst$ and the
inner ones for $\Mlt$.

We can now describe the `draw-and-delete' channel, by post-composing
$\draw$ with the second marginal:
\[ \xymatrix@C+1pc{
\Mlt[K+1](X)\ar[rr]^-{\drawdelete \,\coloneqq\, \Dst(\pi_{2}) \after \draw} 
   & & \Dst\Big(\Mlt[K](X)\Big)
} \]

\noindent Explicitly, $\drawdelete(\varphi) =
\sum_{x} \frac{\varphi(x)}{\sum_{y}\varphi(y)}\bigket{\varphi-1\ket{x}}$. 

Now the consistency property requires that the beliefs about the coin comprise
a sequence of distributions $\omega_K\in \Dst(\Mlt[K](2))$ such that
${(\drawdelete\dnib \omega_{K+1})=\omega_K}$. 
Thus they should form a cone over the sequence
\begin{equation}\label{eqn:chain}\xymatrix{
    \Mlt[0](2) &
    \Mlt[1](2) \ar[l]|-\circ_-\drawdelete&
    \Mlt[2](2) \ar[l]|-\circ_-\drawdelete&
    \dots \ar[l]|-\circ_-\drawdelete&
    \Mlt[K](2) \ar[l]|-\circ_-\drawdelete&\ar[l]|-\circ_-\drawdelete\dots}
\end{equation}

\noindent This means that for all $K$, 
\begin{equation}
  \label{eqn:drawdeleteconecondition}
  \omega_K(\varphi)\ =\ \textstyle \frac {\varphi(B)+1}{K+1}\cdot \omega_{K+1}(\varphi+\ket B) + \frac{\varphi(W)+1}{K+1}\cdot \omega_{K+1}(\varphi+\ket W)\text.
\end{equation}
\paragraph{Known bias is a consistent belief.}
If we know for sure that the bias of a coin is $r\in [0,1]$, then we
would put $\omega_K=\binomial[K](r)$. This is a consistent belief,
which is generalised in the following result to the multinomial
case. This cone forms the basis for this paper.
\begin{proposition}
\label{BinomChainProp}
The multinomial channels form a cone for the chain of draw-and-delete
channels, as in:
\begin{equation}
\label{BinomChainDiag}
\vcenter{\xymatrix{
\Mlt[0](X) 
   & \;\cdots\;\ar[l]|-{\circ}_-{\drawdelete} 
   & \Mlt[K](X)\ar[l]|-{\circ}_-{\drawdelete}  
   & \Mlt[K+1](X)\ar[l]|-{\circ}_-{\drawdelete}
   & \;\cdots\;\ar[l]|-{\circ}_-{\drawdelete}
\\
& & \Dst(X)\ar[u]|-{\circ}^-{\multinomial[K]}
   \ar@/_1ex/[ur]|-{\circ}_-{\quad\multinomial[K+1]\rlap{\qquad$\cdots$}}
   \ar@/^2ex/[ull]|-{\circ}^-{\multinomial[0]\quad}
}}
\end{equation}
\end{proposition}

Our reformulation of the de Finetti theorem explains the sense in which 
above cone is a limit when $X=2$, see Theorem~\ref{FinettiThm}.

\begin{proof}
For $\omega\in\Dst(X)$ and $\varphi\in\Mlt[K](X)$ we have:
\[ \begin{array}[b]{rcl}
\lefteqn{\big(\drawdelete \klafter \multinomial[K+1]\big)(\omega)(\varphi)}
\\[+0.2em]
& = &
{\displaystyle\sum}_{\psi}\, \multinomial[K+1](\omega)(\psi) \cdot
   \drawdelete[K](\psi)(\varphi)
\\[+0.5em]
& = &
{\displaystyle\sum}_{x}\, \multinomial[K+1](\omega)(\varphi + 1\ket{x}) \cdot
   \frac{\varphi(x) + 1}{K+1}
\\[+0.5em]
& = &
{\displaystyle\sum}_{x}\, \frac{(K+1)!}{\prod_{y} (\varphi + 1\ket{x})(y)!}
   \cdot {\textstyle{\displaystyle\prod}_{y}} \omega(y)^{(\varphi + 1\ket{x})(y)} 
   \cdot \frac{\varphi(x) + 1}{K+1}
\\[+0.8em]
& = &
{\displaystyle\sum}_{x}\, \frac{K!}{\prod_{y} \varphi(y)!} \cdot 
   {\textstyle{\displaystyle\prod}_{y}} \omega(y)^{\varphi(y)} \cdot \omega(x) 
\\[+0.8em]
& = &
\displaystyle \frac{K!}{\prod_{y} \varphi(y)!} \cdot 
   \textstyle{\displaystyle\prod}_{y} \omega(y)^{\varphi(y)} \cdot 
   \big(\sum_{x} \omega(x)\big)
\\
& = &
\multinomial[K](\omega)(\varphi).
\end{array}\eqno{\QEDbox} \]
\end{proof}

\paragraph{Unknown bias (uniform) is a consistent belief.}
If we do not know the bias of our coin, it may be reasonable to say that all the $(K+1)$ outcomes in
$\Mlt[K](2)$ are equiprobable. 
This determines the discrete uniform distribution $\upsilon_K\in \Dst(\Mlt[K](2))$,
\[ \begin{array}{rcl}
  \upsilon_K
& \coloneqq &
\displaystyle \sum_{k=0}^K \textstyle \frac 1 {K+1} \bigket{k\ket B+(K-k)\ket W}.
\end{array} \]

\begin{proposition}
\label{UniChainProp}
The discrete uniform distributions $\upsilon_K$ form a cone for the
chain of draw-and-delete channels, as in:
\begin{equation}
\label{UniChainDiag}
\vcenter{\xymatrix{
\Mlt[0](X) 
   & \;\cdots\;\ar[l]|-{\circ}_-{\drawdelete} 
   & \Mlt[K](X)\ar[l]|-{\circ}_-{\drawdelete}  
   & \Mlt[K+1](X)\ar[l]|-{\circ}_-{\drawdelete}
   & \;\cdots\;\ar[l]|-{\circ}_-{\drawdelete}
\\
& & 1\ar[u]|-{\circ}^-{\upsilon_K}
   \ar@/_1ex/[ur]|-{\circ}_-{\quad\upsilon_{K+1}\rlap{\qquad$\cdots$}}
   \ar@/^2ex/[ull]|-{\circ}^-{\upsilon_{0}\quad}
}}
\end{equation}
\end{proposition}

\begin{proof}
For $\varphi\in\Mlt[K](2)$ we have:
\[ \begin{array}[b]{rcl}
\big(\drawdelete \dnib \upsilon_{K+1}\big)(\varphi)
& = &
{\displaystyle\sum}_{\psi}\, \upsilon_{K+1}(\psi) \cdot
   \drawdelete[K](\psi)(\varphi)
\\[+0.2em]
& = &
\upsilon_{K+1}(\varphi+\ket W) \cdot
   \drawdelete[K](\varphi+\ket W)(\varphi)
      \; +
\\
& & \qquad
\upsilon_{K+1}(\varphi+\ket B) \cdot
   \drawdelete[K](\varphi + \ket B)(\varphi)
\\[+0.5em]
& = &
      \frac 1 {K+2}\cdot
      \frac {\varphi(W)+1}{K+1}
      +
      \frac 1 {K+2}\cdot
      \frac {\varphi(B)+1}{K+1}
\\[+0.5em]
& = &
      \frac {\varphi(W)+\varphi(B)+2} {(K+2)(K+1)}
\\[+0.2em]
& = &
      \frac {1} {K+1}\qquad\text{since $\varphi(W)+\varphi(B)=K$.}
      \end{array} \eqno{\QEDbox} \]
\end{proof}

\section{Coalgebras and P\'olya's urn}\label{PolyaSec}

In the previous section we have seen two examples of cones over the
chain~\eqref{eqn:chain}: one for a coin with known bias
(Proposition~\ref{BinomChainProp}), and one where the bias of the coin
is uniformly distributed (Proposition~\ref{UniChainProp}).

In practice, it is helpful to describe our belief about the coin in
terms of a data generating process.  This is formalised as a channel
$h\colon X\chanto 2\times X$, \ie~as a coalgebra for the functor
$2\times (-)$ on the category of channels.  The idea is that our
belief about the coin is captured by some parameters, which are the
states of the coalgebra~$x\in X$. The distribution $h(x)$ describes
our belief about the outcome of the next coin toss, but also provides
a new belief state according to the outcome. For example, if we have
uniform belief about the bias of the coin, we predict the outcome to
be heads or tails equally, but afterward the toss our belief is
refined: if we saw heads, we believe the coin to be more biased
towards heads.

P\'olya's urn is an important example of such a coalgebra.  We
describe it as a multiset over $2 = \{0,1\}$ with a certain extraction
and re-fill operation. The number $0$ corresponds to black balls, and
$1$ to white ones, so that a multiset $\varphi\in\Mlt(2)$ captures an
urn with $\varphi(0)$ black and $\varphi(1)$ white ones. When a ball
of color $x\in 2$ is extracted, it is put back together with an extra
ball of the same color.  We choose to describe it as a coalgebra on
the set $\neMlt(2)$ of non-empty multisets over $2$.  Explicitly,
$\textsl{pol}\colon \neMlt(2)\chanto 2\times \neMlt(2)$ is given by:
\[ \begin{array}{rcl}
\textsl{pol}(\varphi)
& \coloneqq &
\displaystyle
\frac{\varphi(0)}{\varphi(0)+\varphi(1)}\bigket{0, \varphi + 1\ket{0}} +
   \frac{\varphi(1)}{\varphi(0)+\varphi(1)}\bigket{1, \varphi + 1\ket{1}}.
   \end{array} \]

\noindent So the urn keeps track of the outcomes of the coin tosses so
far, and our belief about the bias of the coin is updated according to
successive experiments.

\subsection{Iterating draws from coalgebras}\label{sec:iteration-coalg}

By iterating a coalgebra $h\colon X\chanto 2\times X$, we can simulate
the experiment with $K$ coin tosses $h_K\colon X\chanto \Mlt[K](2)$,
by repeatedly applying~$h$.  Formally, we build a sequence of channels
$h^K\colon X\chanto \Mlt[K](2)\times X$ by $h^0(x)\coloneqq(\zero,x)$
and:
\[  h^{K+1}\coloneqq\xymatrix@C-3pt{X\ar[r]|-\circ^-{h^K}& \Mlt[K](2)\times X\ar[r]|-\circ ^-{\idmap\times h}& 
   \Mlt[K](2)\times 2\times X\ar[r]^{\mathsl{add}\times\idmap}&\Mlt[K+1](2)\times X}\]

\noindent where the rightmost map $\mathsl{add}\colon \Mlt[K](2)\times
2\to \Mlt[K+1](2)$ is $\mathsl{add}(\varphi,b)=\varphi+1\ket b$.

(Another way to see this is to regard a coalgebra $h\colon X\chanto
2\times X$ in particular as coalgebra $X\chanto \Mlt(2)\times X$. Now
$\Mlt(2)\times -$ is the writer monad for the commutative monoid
$\Mlt(2)$, i.e.~the monad for $\Mlt(2)$-actions. Iterated Kleisli composition for this monad gives a
sequence of maps $h\colon X\chanto \Mlt(2)\times X$.)

For example, starting with P\'olya's urn $\textsl{pol}\colon
\neMlt(2)\chanto 2\times \neMlt(2)$, we look at a few iterations:
\begin{equation}\label{eqn:polyasteps} 
\begin{array}{rcl}
\big(\mathsl{pol}\big)^{0}(b\ket{0} + w\ket{1})
& = &
1\bigket{\zero, b\ket{0} + w\ket{1}}
\\
\big(\mathsl{pol}\big)^{1}(b\ket{0} + w\ket{1})
& = &
\mathsl{pol}(b\ket{0} + w\ket{1})
\\
& = &
\frac{b}{b+w}\bigket{1\ket{0}, \varphi + \ket{0}} +
   \frac{w}{b+w}\bigket{1\ket{1}, \varphi + \ket{1}}
\\
\big(\mathsl{pol}\big)^{2}(b\ket{0} + w\ket{1})
& = &
\frac{b}{b+w}\cdot\frac{b+1}{b+w+1}\bigket{2\ket{0}, (b+2)\ket{0} + w\ket{1}}
\\
& & \; + \,
\frac{b}{b+w}\cdot\frac{w}{b+w+1}\bigket{1\ket{0} + 1\ket{1}, 
   (b+1)\ket{0} + (w+1)\ket{1}}
\\
& & \; + \,
\frac{w}{b+w}\cdot\frac{b}{b+w+1}\bigket{1\ket{0} + 1\ket{1}, 
   (b+1)\ket{0} + (w+1)\ket{1}}
\\
& & \; + \,
\frac{w}{b+w}\cdot\frac{w+1}{b+w+1}\bigket{2\ket{1}, b\ket{0} + (w+2)\ket{1}}
\\[+0.4em]
& = &
\frac{b(b+a)}{(b+w)(b+w+1)}\bigket{2\ket{0}, (b+2)\ket{0} + w\ket{1}}
\\
& & \; + \,
\frac{2bw}{(b+w)(b+w+1)}\bigket{1\ket{0} + 1\ket{1}, 
   (b+1)\ket{0} + (w+1)\ket{1}}
\\
& & \; + \,
\frac{w(w+1)}{(b+w)(b+w+1)}\bigket{2\ket{1}, b\ket{0} + (w+2)\ket{1}}.
\end{array} \end{equation}

\noindent It is not hard to see that we get as general formula,
for $K\in\NNO$:
\[ \begin{array}{rcl}
\big(\mathsl{pol}\big)^{K}(b\ket{0} + w\ket{1})
& = &
\displaystyle\!\!\sum_{0 \leq k \leq K} \textstyle
   \!\binom{K}{k} \cdot \!\displaystyle
   \frac{\prod_{i<k}(b+i) \cdot \prod_{i<K-k}(w+i)}{\prod_{i<K}(b+w+i)}
\\
& & \hspace*{0.5em}
   \bigket{k\ket{0} \!+\! (K-k)\ket{1}, \, 
   (b \!+\! k)\ket{0} + (w \!+\! (K\!-\!k))\ket{1}}.
\end{array} \]

\noindent We are especially interested in the first component of
$h^K(x)$, which is in $\Mlt[K](2)$. Thus we define the channels:
\[ h_K\,\coloneqq\,
  \Big(\xymatrix@R-0.5pc{
X\ar[r]^-{\mathsl{h}^K}|-{\circ} 
& \Mlt[K](2)\times X
\ar[r]^-{\pi_1}&
 \Mlt[K](2)}\Big)\]

For the P\'olya urn $\mathsl{pol}$, the really remarkable fact about
$\mathsl{pol}_K\colon \neMlt(2)\chanto \Mlt[K](2)$ is that if we start
with the urn with one white ball and one black ball, these
distributions on $\Mlt[K](2)$ that come from iteratively drawing are
uniform distributions:
\[ \begin{array}{rcl}
\mathsl{pol}_K(1\ket 0 + 1\ket 1)
& = &
\upsilon_K\text.
\end{array} \]

More generally:

\begin{proposition}
\label{PolyaIteratedProp}
\begin{equation}
\label{PolyaIteratedDivEqn}
\begin{array}{rcl}
\lefteqn{\mathsl{pol}_K(b\ket 0 + w\ket 1)\big(k\ket{0} + (K-k)\ket{1}\big)}
\\[+0.3em]
& = &
\binom{K}{k} \cdot \displaystyle
   \frac{(b + k - 1)!}{(b- 1)!} \cdot
   \frac{(w+ K - k - 1)!}{(w- 1)!} \cdot
   \frac{(b+ w- 1)!}
   {(b+ w+ K - 1)!}.
\end{array}
\end{equation}
\end{proposition}

\begin{proof}
We start from the $K$-th iteration formula for $\mathsl{pol}$, given
just before Proposition~\ref{PolyaIteratedProp}, and derive:
\[ \begin{array}[b]{rcl}
\lefteqn{\displaystyle
\frac{\prod_{i<k}(b+i) \cdot \prod_{i<K-k}(w+i)}{\prod_{i<K}(b+w+i)}}
\\[+1em]
& = &
\displaystyle
   \frac{(b+ k - 1)!}{(b- 1)!} \cdot
   \frac{(w+ K -k - 1)!}{(w - 1)!} \cdot
   \frac{(b+ w- 1)!}
   {(b+ w+ K - 1)!}.
\end{array} \eqno{\QEDbox} \]
\end{proof}
From this characterization, we can deduce that
the channels $\mathsl{pol}_K\colon \neMlt(2)\chanto \Mlt[K](2)$
form a cone over the draw-and-delete chain \eqref{eqn:chain}, \ie~that
$\drawdelete\klafter \mathsf{pol}_{K+1}=\mathsf{pol}_K$, 
by a routine calculation.

In what follows we provide two other ways to see that $\mathsl{pol}_K$
forms a cone.  First we deduce it using the fact that the coalgebra
$\mathsl{pol}$ is exchangeable (Lemma~\ref{ExchCoalgLem}).  Second we
deduce it by noticing that $\mathit{pol}_K(b\ket 0+w\ket 1)(k\ket
0+(K-k)\ket 1)$ is distributed as $\betabinomial[K] (b,w)$, the beta
binomial distribution, and hence factors through the binomial cone
(Proposition~\ref{BinomChainProp}) via the beta channel
(Lemma~\ref{BetaBinomLem}).

\subsection{Exchangeable coalgebras and cones}

We would like to focus on those coalgebras $h\colon X\chanto 2\times
X$ that describe consistent beliefs, \ie~for which $h_K\colon X\chanto
\Mlt[K](2)$ is a cone over the chain~\eqref{eqn:chain}:
\begin{equation}\label{eqn:cone}
  \drawdelete \klafter h_{K+1}=h_K\text.
\end{equation}

\noindent For example, P\'olya's urn forms a cone.  But the
deterministic alternating coalgebra
\begin{equation}
  \label{eqn:altcoalg}
  \mathsl{alt}:2\chanto 2\times 2\qquad 
  \mathsl{alt}(b)\coloneqq 1\ket {b,\neg b}
\end{equation}

\noindent does not form a cone, since for instance
$\mathsl{alt}_1(0)=1\ket 0$, while $\drawdelete\klafter
\mathsl{alt}_2(0)=\frac 12 \ket 0+\frac 12 \ket 1$.

The following condition is inspired by de Finetti's notion of
exchangeability, and gives a simple coinductive criterion for a
coalgebra to yield a cone~\eqref{eqn:cone}. It is also related to
ideas from probabilistic programming (\eg~\cite{Xrp2016}).
\begin{definition}
\label{ExchCoalgDef}
A coalgebra $h\colon X\chanto 2\times X$ will be called \emph{exchangeable} if
its outputs can be reordered without changing the statistics, as
expressed by commutation of the following diagram.
\begin{equation}
\label{eqn:exch}
\vcenter{\xymatrix@R-1.5pc@C+1pc{
& 2\times X\ar[r]^-{\idmap\otimes h}|-{\circ} & 
   2\times 2\times X\ar[dd]^-{\mathsl{swap}\times\idmap[X]}_-{\cong}
\\
X\ar@/^2ex/[ur]^-{h}|-{\circ}\ar@/_2ex/[dr]_-{h}|-{\circ}
\\
& 2\times X\ar[r]_-{\idmap\otimes h}|-{\circ} & 2\times 2\times X
}}
\end{equation}
\end{definition}

Returning to P\'olya's urn coalgebra $\mathsl{pol}\colon
\neMlt(2)\chanto 2\times \neMlt(2)$, we see that it is exchangeable,
for, by a similar argument to~\eqref{eqn:polyasteps},
\[\begin{array}{rcl}
    (\idmap\times \mathsl{pol})(\mathsl{pol}(b\ket 0+w\ket 1))=
    & = &
          \frac{b}{b+w}\cdot\frac{b+1}{b+w+1}\bigket{(0,0), (b+2)\ket{0} + w\ket{1}}
    \\
& & \; + \,
    \frac{b}{b+w}\cdot\frac{w}{b+w+1}\bigket{(0,1), 
    (b+1)\ket{0} + (w+1)\ket{1}}
    \\
    & & \; + \,
        \frac{w}{b+w}\cdot\frac{b}{b+w+1}\bigket{(1,0), 
        (b+1)\ket{0} + (w+1)\ket{1}}
    \\
    & & \; + \,
        \frac{w}{b+w}\cdot\frac{w+1}{b+w+1}\bigket{(1,1), b\ket{0} + (w+2)\ket{1}}
\end{array}\]

\noindent Reordering the results does not change the statistics,
because $b\cdot w=w\cdot b$.  On the other hand, the deterministic
alternating coalgebra $\mathsl{alt}\colon 2\chanto 2\times 2$
from~\eqref{eqn:altcoalg} is not exchangeable, since:
\[\begin{array}{rcl}
  (\idmap\times \mathsl{alt})(\mathsl{alt}(0))&=&
                                                 \ket{((0,1),0)}\\
  (\mathsl{swap}\times 2)((\idmap\times \mathsl{alt})(\mathsl{alt}(0)))&=&
                                                 \ket{((1,0),0)}                                           \text.      
  \end{array}\]

Before we state and prove our theorem about exchangeable coalgebras,
we also consider a different sequence, which keeps track of the order
of draws.  We define a sequence $h^{\sharp K}\colon X\chanto 2^K\times
X$, with $h^{\sharp 0}(x)=((),x)$ and $h^{\sharp (K+1)}(x)
=(2^K\otimes h)\dnib h^{\sharp K}(x)$.  Then we define channels
$h^\sharp_K:X\chanto 2^K$ by
\begin{equation}
  \label{eqn:channelsequence}
  h^\sharp_K(x)
  \coloneqq \Dst(\pi_{1})(h^{\sharp K}(x))
\end{equation}

\noindent These constructions are related similar to the earlier
constructions $h^K$ and $h_K$ by
\[ \xymatrix@R-20pt{
h^K \;=\; \Big(X\ar[r]^-{h^{\sharp K}}|-{\circ} &
2^{K}\times X\ar[r]^-{\mathsl{acc}\times X} & \Mlt[K](2)\times X\Big)
\\
    h_{K} \;=\; \Big(X\ar[r]^-{h^{\sharp}_{K}}|-{\circ} &
   2^{K}\ar[r]^-{\mathsl{acc}} & \Mlt[K](2)\Big),
} \]
\noindent where $\mathsl{acc} \colon 2^{K} \rightarrow \Mlt[K](2)$ is
the accumulator map from~\eqref{AccDiag}.  The difference is that the
$h^{\sharp}$ variants keep track of the order of draws, which will be
useful in our arguments about exchangeability, since $h$ is
exchangeable (Definition~\ref{ExchCoalgDef}) if and only if
$h^{\sharp2}=\Dst(\mathsl{swap}\times X)\circ h^{\sharp2}:X\chanto
2\times 2\times X$.

\begin{lemma}
\label{ExchCoalgLem}
Let $h\colon X \chanto 2\times X$ be an exchangeable coalgebra. 
Then 
\begin{enumerate}
\item the map $h_{K}$ can be expressed as:
\begin{equation}
\label{ExchCoalgEqn}
\begin{array}{rcl}
h_{K}(x)
& = &
\displaystyle\sum_{k=0}^{K} \textstyle \binom{K}{k}\cdot 
   h^{\sharp}_{K}(x)\big(\,\underbrace{1,\ldots,1}_{k\text{ times}}, \,
   \underbrace{0,\ldots,0}_{K-k\text{ times}}\big) \bigket{k\ket 1+(K-k)\ket 0}.
\end{array}
\end{equation}

\item the collection of maps $(h_{K})$ forms a cone for the chain of
  draw-and-delete maps $\drawdelete$ \eqref{eqn:chain}.
\end{enumerate}
\end{lemma}

\begin{proof}
  \begin{enumerate}
\item Among $b_{1}, \ldots, b_{K} \in 2$ there are $\binom{K}{k}$-many
  possibilities of having $k$-many ones. By exchangeability of $h$
  there is a normal form with all these ones upfront
\[ \begin{array}{rcl}
h^{\sharp}_{K}(x)\big(b_{1},\ldots, b_{K}\big)
& = &
h^{\sharp}_{K}(x)\big(\,\underbrace{1,\ldots,1}_{k\text{ times}}, \,
   \underbrace{0,\ldots,0}_{K-k\text{ times}}\big)
\end{array} \]

\noindent 

\item Using~\eqref{eqn:drawdeleteconecondition} we get, for any $\varphi= (k\ket 1+(K-k)\ket 0) \in \Mlt[K](2)$,
\[ \hspace*{-0.5em}\begin{array}[b]{rcl}
\lefteqn{\big(\drawdelete \klafter h_{K+1}\big)(x)(\varphi)}
\\
& = &
\frac{k+1}{K+1} \cdot h_{K+1}(x)(\varphi+\ket 1) + 
   \frac{K+1-k}{K+1} \cdot h_{K+1}(x)(\varphi +\ket 0)
\\[+0.2em]
& \smash{\stackrel{\eqref{ExchCoalgEqn}}{=}} &
\frac{k+1}{K+1} \cdot \binom{K+1}{k+1}\cdot 
   h^{\sharp}_{K+1}(x)\big(\underbrace{1,\ldots,1}_{k+1\text{ times}}, 0,\ldots,0\big)
  \; +
\\
& & \qquad
\frac{K+1-k}{K+1} \cdot \binom{K+1}{k}\cdot 
   h^{\sharp}_{K+1}(x)\big(\underbrace{1,\ldots,1}_{k\text{ times}}, 0,\ldots,0\big)
\\
& = &
\binom{K}{k}\cdot\Big(
h^{\sharp}_{K+1}(x)\big(\underbrace{1,\ldots,1}_{k\text{ times}}, 0,\ldots,0,1\big)
+
h^{\sharp}_{K+1}(x)\big(\underbrace{1,\ldots,1}_{k\text{ times}}, 0,\ldots,0\big)\Big)
\\
& = &
\binom{K}{k}\cdot \sum_{y} 
   h^{\sharp K}(x)\big(\underbrace{1,\ldots,1}_{k\text{ times}}, 0,\ldots,0,y\big)
   \cdot \Big(h^{\sharp}_{1}(y)(1) + h^{\sharp}_{1}(y)(0)\Big)
\\
& = &
\binom{K}{k}\cdot \sum_{y} 
   h^{\sharp K}(x)\big(\underbrace{1,\ldots,1}_{k\text{ times}}, 0,\ldots,0,y\big)
   \cdot 1
\\
& = &
\binom{K}{k}\cdot
   h^{\sharp}_{K}(x)\big(\underbrace{1,\ldots,1}_{k\text{ times}}, 0,\ldots,0\big)
\\
& \smash{\stackrel{\eqref{ExchCoalgEqn}}{=}} &
h_{K}(x)(\varphi).
\end{array} \eqno{\QEDbox} \]
\end{enumerate}
\end{proof}





Since the P\'olya's urn coalgebra $\mathsl{pol}$ is exchangeable, the
lemma provides another way to deduce that the channels
$\mathsl{pol}_K\colon \neMlt(2)\chanto \Mlt[K](2)$ form a
cone~\eqref{eqn:cone}.

\section{Background on continuous probability}

De Finetti's theorem is about the passage from discrete probability to
continuous probability.  Consider the uniform distribution $\mu$ on
$[0,1]$. The probability of any particular point is $0$.  So rather
than give a probability to individual points, we give probabilities to
sets; for example, the probability of a point in the half interval
$[0,\frac 12]$ is $\frac 12$.

In general, a $\sigma$-algebra on a set $X$ is a collection of sets
$\Sigma$ that is closed under countable unions and complements. A
measurable space $(X,\Sigma)$ is a set together with a
$\sigma$-algebra; the sets in $\Sigma$ are called measurable. In this
paper we mainly consider two kinds of $\sigma$-algebra, and so they
will often be left implicit:
\begin{itemize}
\item on a countable set (\eg~$2$, $\Mlt(2)$, $\Mlt[K](2)$) we
  consider the $\sigma$-algebra where every set is measurable.

\item on the unit interval $[0,1]$, the Borel $\sigma$-algebra, which
  is the least $\sigma$-algebra containing the intervals.
\end{itemize}
On a measurable space $(X,\Sigma)$, a \emph{probability
  measure} is a function $\omega\colon\Sigma \rightarrow [0,1]$
with $\omega(X) = 1$ and with $\omega(\bigcup_{i}U_{i}) =
\sum_{i}\omega(U_{i})$ when the measurable subsets $U_{i}\in\Sigma$
are pairwise disjoint.  On a
finite set $X$ the probability measures coincide with the discrete distributions.

The morphisms of measurable spaces are the functions $f\colon X\to Y$
for which the inverse image of a measurable set is
measurable. Crucially, if $f\colon X\to [0,1]$ is measurable and
$\omega \colon \Sigma\to [0,1]$ is a probability measure, then we can
find the Lebesgue integral $\int f(x)\,\omega(\dd x)$, which is the
expected value of~$f$. When $X$ is finite, the Lebesgue integral is a
sum: $\int f(x)\,\omega(\dd x)=\sum_Xf(x)\omega(\{x\})$. When
$X=[0,1]$ and $\omega$ is the uniform distribution, then the Lebesgue
integral is the familiar one.

We write $\Giry(X,\Sigma)$ for the set of all probability measures on
$(X,\Sigma)$, where we often omit the $\sigma$-algebra $\Sigma$ when
it is clear from the context.: $\Dst(X) \cong \Giry(X)$, when all
subsets are seen as measurable.  The set $\Giry(X,\Sigma)$ itself has
a $\sigma$-algebra, which is the least $\sigma$-algebra making the
sets $\{\omega~|~\omega(U)<r\}$ measurable for all fixed $U$ and $r$.

Like $\Dst$, the construction $\Giry$ is a monad, commonly called the
Giry monad, but $\Giry$ is a monad on the category of measurable
spaces.  Maps in the associated Kleisli category $\Kl(\Giry)$ will
also be called \emph{channels}. Thus $f\colon X \chanto Y$ between
measurable spaces can be described as a function ${f\colon X \times
  \Sigma_Y\rightarrow [0,1]}$ that is a measure in $\Sigma_Y$ and
measurable in $X$.  For a measure~$\omega\in\Giry(X)$ and a channel
$g\colon Y \chanto Z$ we define a new distribution $f \dnib \omega \in
\Giry(Y)$ and a new channel $g \klafter f \colon X \chanto Z$ by
Kleisli composition:
\[ \begin{array}{rclcrcl}
(f \dnib \omega)(U)
& \coloneqq &
\displaystyle\textstyle
              \int_{X}  f(x)(U)\,\omega(\dd x)
& \qquad &
\big(g \klafter f\big)(x)
& \coloneqq &
g \dnib f(x).
\end{array} \]
 
Notice that a discrete channel is in particular a continuous channel, according to our terminology,
and so we have an inclusion of Kleisli categories categories
$\Kl(\Dst) \hookrightarrow \Kl(\Giry)$.

\paragraph{Illustrations from the beta-bernoulli distributions.}

At the end of Section~\ref{sec:iteration-coalg}, we briefly mentioned
the beta-bernoulli distribution as it arises from the coalgebra for
P\'olya's urn.  In general, this beta-binomial channel is defined for
arbitrary $\alpha,\beta\in\pR$ as:
\begin{equation}
\label{BetabinDefEqn}
\hspace*{-1em}\begin{array}{rcl}
\betabinomial[K](\alpha,\beta)
& \coloneqq &
\displaystyle\!\!\sum_{0\leq k\leq K} \textstyle \binom{K}{k} \cdot \displaystyle
    \frac{\Beta(\alpha + k, \beta + K-k)}{\Beta(\alpha, \beta)}
    \bigket{k\ket{0} \!+\! (K-k)\ket{1}},
\end{array}
\end{equation}

\noindent This description involves the function $\Beta \colon \pR
\times \pR \rightarrow \R$ defined as:
\begin{equation}
\label{BetaDefEqn}
\begin{array}{rcl}
\Beta(\alpha,\beta)
& \coloneqq &
\displaystyle\int_{0}^{1} x^{\alpha-1}\cdot (1-x)^{\beta-1}\intd x
\end{array}
\end{equation}

\noindent This $\Beta$ is a well-known mathematical function
satisfying for instance:
\begin{equation}
\label{BetaEqns}
\begin{array}{rclcrcl}
\Beta(\alpha+1,\beta)
& = &
\displaystyle\frac{\alpha}{\alpha+\beta}\cdot\Beta(\alpha,\beta)
& \qquad &
\Beta(n+1,m+1)
& = &
\displaystyle\frac{n!\cdot m!}{(n+m+1)!},
\end{array}
\end{equation}

\noindent for $n,m\in\NNO$. The latter equation shows
that~\eqref{PolyaIteratedDivEqn} is an instance of~\eqref{BetabinDefEqn}.

The function $\Beta$ is also used to define the Beta distribution on
the unit interval $[0,1]$. We describe it here as a channel in the
Kleisli category of the Giry monad:
\[ \xymatrix{
\pR\times\pR\ar[rr]^-{\betachan} & & \Giry([0,1]).
} \]

\noindent It is defined on $\alpha,\beta\in\pR$ and measurable
$M\subseteq [0,1]$ as:
\begin{equation}
\label{BetaChanDefEqn}
\begin{array}{rcl}
\betachan(\alpha,\beta)(M)
& \coloneqq &
\displaystyle
   \int_{M} \frac{x^{\alpha-1}\cdot (1-x)^{\beta-1}}{\Beta(\alpha,\beta)}\intd x.
\end{array}
\end{equation}

We prove some basic properties of the beta-binomial channel. These
properties are well-known, but are formulated here in slightly
non-standard form, using channels.

\begin{lemma}
\label{BetaBinomLem}
\begin{enumerate}
\item \label{BetaBinomLemComp} A betabinomial channel can itself be
decomposed as:
\[ \xymatrix@R-0.5pc{
& \Mlt[K](2)
\\
\pR \times \pR\ar[ur]^-{\betabinomial[K]\quad}|-{\circ}
   \ar[r]_-{\betachan}|-{\circ}
  & [0,1]\ar[u]_-{\binomial[K]}|-{\circ}
} \]

\item \label{BetaBinomLemCone} The beta-binomial channels
\[ \xymatrix@C+2pc{
\pR\times\pR\ar[r]^-{\betabinomial[K]}|-{\circ} & \Mlt[K](2)
} \]

\noindent for $K\in\NNO$, form a cone for the infinite chain of
draw-and-delete channels \eqref{eqn:chain}.
\end{enumerate}
\end{lemma}

\begin{proof}
\begin{enumerate}
\item Composition in the Kleisli category $\Kl(\Giry)$ of the Giry
  monad $\Giry$ yields:
\[ \begin{array}{rcl}
\lefteqn{\big(\binomial[K] \klafter \betachan\big)
   (\alpha,\beta)\big(k\ket{0} \!+\! (K-k)\ket{1}\big)}
\\
& = &
\displaystyle\int_{0}^{1} {\textstyle\binom{K}{k}} \cdot 
   x^{k}\cdot (1-x)^{K-k} \cdot
   \frac{x^{\alpha-1} \cdot (1-x)^{\beta-1}}{\Beta(\alpha,\beta)} \intd x
\\[+0.8em]
& = &
\binom{K}{k} \cdot \displaystyle
   \frac{\int_{0}^{1} x^{\alpha+k-1} \cdot (1-x)^{\beta + (K-k)-1} \intd x}
   {\Beta(\alpha,\beta)} 
\\[+0.8em]
& \smash{\stackrel{\eqref{BetaDefEqn}}{=}} &
\binom{K}{k} \cdot \displaystyle
   \frac{\Beta(\alpha+k, \beta+K-k)}{\Beta(\alpha,\beta)}
\\[+0.6em]
& \smash{\stackrel{\eqref{BetabinDefEqn}}{=}} &
\betabinomial[K](\alpha,\beta)\big(k\ket{0} \!+\! (K-k)\ket{1}\big).
\end{array} \]

\item The easiest way to see this is by observing that the
  channels factorise via the binomial channels
  $\binomial[K] \colon [0,1] \chanto \Mlt[K](2)$. In
  Proposition~\ref{BinomChainProp} we have seen that these binomial
  channels commute with draw-and-delete channels $\drawdelete$. \QED
\end{enumerate}
\end{proof}

\section{De Finetti as a limit theorem}\label{FinettiSec}

We first describe our `limit' reformulation of the de Finetti theorem
and later on argue why this is a reformulation of the original result.

\begin{theorem}
\label{FinettiThm}
For $X = 2$ the cone~\eqref{BinomChainDiag} is a limit in the Kleisli
category $\Kl(\Giry)$ of the Giry monad; this limit is the unit
interval $[0,1] \cong \Dst(2)$.

This means that if we have channels $c_{K} \colon Y \chanto \Mlt[K](2)
\cong \{0,1,\ldots,K\}$ with $\drawdelete \klafter c_{K+1} = c_{K}$,
then there is a unique mediating channel $c \colon Y \chanto [0,1]$
with $\binomial[K] \klafter c = c_{K}$, for each $K\in\NNO$.
\end{theorem}

In the previous section we have seen an example with $Y = \pR \times
\pR$ and $c_{K} = \betabinomial[K] \colon Y \chanto \Mlt[K](2)$ and
mediating channel $c = \betachan \colon Y \chanto [0,1]$.

Below we give the proof, in the current setting. We first illustrate
the central role of the draw-and-delete maps $\drawdelete \colon
\Mlt[K+1](X) \chanto \Mlt[K](X)$ and how they give rise to
hypergeometric and binomial distributions.

Via the basic isomorphism $\Mlt[K](2) \cong \{0,1,\ldots,K\}$ the
draw-and-delete map~\eqref{eqn:drawdeleteconecondition} becomes a map
$\drawdelete \colon \{0,\ldots,K+1\} \chanto \{0,\ldots,K\}$, given
by:
\begin{equation}
\label{DrawdeleteBinEqn}
\begin{array}{rcl}
\drawdelete(\ell)
& = &
\frac{\ell}{K+1}\ket{\ell-1} + \frac{K+1-\ell}{K+1}\ket{\ell}.
\end{array}
\end{equation}

\noindent Notice that the border cases $\ell = 0$ and $\ell = K+1$ are
handled implicitly, since the corresponding probabilities are then
zero.

We can iterate $\drawdelete$, say $N$ times, giving $\drawdelete^{N}
\colon \{0,\ldots,K+N\} \chanto \{0,\ldots,K\}$.  If we
draw-and-delete $N$ times from an urn with $K+N$ balls, the result is
the same as sampling without replacement $K$ times from the same urn.
Sampling without replacement is modelled by the \emph{hypergeometric
  distribution}.  So we can calculate $\drawdelete^N$ explicitly,
using the hypergeometric distribution formula:
\[ \begin{array}{rcl}
\drawdelete^{N}(\ell)
& = &
\displaystyle\sum_{0\leq k\leq K} 
   \frac{\binom{\ell}{k} \cdot \binom{K+N-\ell}{K-k}}{\binom{K+N}{K}}\bigket{k}.
\end{array} \]

\noindent Implicitly the formal sum is over $k\leq\ell$, since
otherwise the corresponding probabilities are zero.

As is well known, for large $N$, the hypergeometric distribution can
be approximated by the binomial distribution.  To be precise, for
fixed $K$, if $(\nicefrac{\ell}{K+N})\to p$ as $N\to\infty$, then the
hypergeometric distribution $\drawdelete^{N}(\ell)$ approaches the
binomial distribution $\binomial[K](p)$ as $N$ goes to infinity.  This
works as follows. Using the notation $(n)_{m} = \prod_{i<m} (n-i) =
n(n-1)(n-2)\cdots(n-m+1)$ we can write $\binom{n}{m} =
\nicefrac{(n)_{m}}{n!}$ and thus:
\[ \begin{array}{rcl}
\displaystyle\frac{\binom{\ell}{k} \cdot \binom{K+N-\ell}{K-k}}{\binom{K+N}{K}}
& = &
\displaystyle\frac{(\ell)_{k}}{k!} \cdot \frac{(K+N-\ell)_{K-k}}{(K-k)!}
   \cdot \frac{K!}{(K+N)_{K}}
\\
& = &
\binom{K}{k}\cdot \displaystyle
   \frac{(\ell)_{k} \cdot (K+N-\ell)_{K-k}}{(K+N)_{K}}
\\[+0.8em]
& = &
\binom{K}{k}\cdot \displaystyle
   \frac{\prod_{i<k} (\ell-i) \cdot \prod_{i<K-k} (K+N-\ell-i)}
   {\prod_{i<K}(K+N-i)}
\\[+0.8em]
& = &
\binom{K}{k}\cdot \displaystyle
   \frac{\prod_{i<k} (\nicefrac{\ell}{K+N}-\nicefrac{i}{K+N}) \cdot 
   \prod_{i<K-k} (1-\nicefrac{\ell}{K+N}-\nicefrac{i}{K+N})}
   {\prod_{i<K}(1-\nicefrac{i}{K+N})}
\\[+0.8em]
& \rightarrow &
\binom{K}{k}\cdot p^{k} \cdot (1-p)^{K-k}
   \qquad \mbox{as $N \rightarrow \infty$}
\\
& = &
\binomial[K](p)(k).
\end{array} \]

The proof of Theorem~\ref{FinettiThm} makes use of a classical result
of Hausdorff~\cite{Hausdorff21}, giving a criterion for the existence
of probability measures. It is repeated here without proof (but see
\eg~\cite{Feller70} for details).

Recall that the \emph{moments} of a distribution $\mu\in\Giry([0,1])$
form a sequence of numbers $\frak{m}_1, \dots,
\frak{m}_K,\ldots\in[0,1]$ given by $\frak{m}_K=\int_0^1 r^K\,\mu(\dd
r)$: the expected values of $r^K$.  In general, moments correspond to
statistical concepts such as mean, variance, skew and so on.  But in
this specific case, the $K$-th moment of $\mu$ is the probability
$(\binomial[K]\dnib\mu)(K)$ that $K$ balls drawn will \emph{all} be
black.

Recall that a sequence $a = (a_{n})$ is \emph{completely monotone} if:
\[ \begin{array}{rcl}
(-1)^{k}\cdot \big(\Delta^{k} a\big)_{n}
& \geq &
0
\qquad \mbox{for all }n,k.
\end{array} \]


\noindent This formulation involves the difference operator $\Delta$
on sequences, given by $(\Delta a)_{n} = a_{n+1} - a_{n}$. This
operator is iterated, in the standard way as $\Delta^{0} = \idmap$ and
$\Delta^{m+1} = \Delta \after \Delta^{m}$.

\begin{theorem}[Hausdorff]
\label{HausdorffThm}
A sequence $a = (a_{n})_{n\in\NNO}$ of non-negative real numbers $a_n$
is \emph{completely monotone} if and only if it arises as sequence of
moments: there is a unique measure $\mu\in\Giry([0,1])$ with, for each
$n$,
\[ \begin{array}{rcl}
a_{n}
& = &
\displaystyle\int_{0}^{1} x^{n} \,\mu(\dd x).
\end{array} \eqno{\QEDbox} \]
\end{theorem}

\begin{proof}[of Theorem~\ref{FinettiThm}]
As a first step, we shall reason pointwise. Let, for each $K\in\NNO$,
a distribution $\omega_{K}\in \Dst(\{0,1,\ldots,K\})$ be given, with
$\drawdelete \dnib \omega_{K+1} = \omega_{K}$. We need to produce a
unique (continuous) distribution $\omega\in\Giry([0,1])$ such that
$\omega_{K} = \binomial[K] \dnib \omega$, for each $K$.

The problem of finding $\omega$ is essentially Hausdorff's moment
problem, see Theorem~\ref{HausdorffThm}. The connection between
Hausdorff's moments problem and de Finetti's theorem is well
known~\cite[\S VII.4]{Feller70}, but the connection is particularly
apparent in our channel-based formulation.


We exploit Hausdorff's criterion via two
claims, which we elaborate in some detail. The first one is standard.

\smallskip

\noindent\textbf{Claim 1}. If $\mu\in\Giry([0,1])$ has moments
$\frak{m}_{K}$, then:
\[ \begin{array}{rcl}
\displaystyle\int_{0}^{1} x^{n}\cdot (1-x)^{k} \,\mu(\dd x)
& = &
(-1)^{k}\cdot \big(\Delta^{k} \frak{m}\big)_{n}.
\end{array} \]

\noindent This equation can be proven by induction on $k$. It holds by
definition of $\frak{m}$ for $k=0$. Next:
\[ \begin{array}{rcl}
\lefteqn{\displaystyle\int_{0}^{1} x^{n}\cdot (1-x)^{k+1} \,\mu(\dd x)}
\\
& = &
\displaystyle\int_{0}^{1} x^{n}\cdot (1-x)\cdot (1-x)^{k} \,\mu(\dd x)
\\[+0.8em]
& = &
\displaystyle\int_{0}^{1} x^{n}\cdot (1-x)^{k} \,\mu(\dd x) \;-\;
   \displaystyle\int_{0}^{1} x^{n+1}\cdot (1-x)^{k} \,\mu(\dd x)
\\[+0.8em]
& \smash{\stackrel{\text{(IH)}}{=}} &
(-1)^{k}\cdot \big(\Delta^{k} \frak{m}\big)_{n} \;-\;
   (-1)^{k}\cdot \big(\Delta^{k} \frak{m}\big)_{n+1}
\\[+0.3em]
& = &
(-1)^{k+1}\cdot \Big(\big(\Delta^{k} \frak{m}\big)_{n+1} - 
   \big(\Delta^{k} \frak{m}\big)_{n}\Big)
\\
& = &
(-1)^{k+1}\cdot \big(\Delta^{k+1} \frak{m}\big)_{n}.   
\end{array} \]




The next claim involves a crucial observation about our setting: the
whole sequence of distributions $\omega_{K} \in
\Dst(\{0,1,\ldots,K\})$ is entirely determined by the probabilities
$\omega_{K}(K)$ of drawing $K$ balls which are all black.

\smallskip

\noindent\textbf{Claim 2}. The numbers $\frak{b}_{K} \coloneqq
\omega_{K}(K)\in [0,1]$, for $K\in\NNO$, determine all the
distributions $\omega_K$ via the relationship:
\[ \begin{array}{rcl}
\omega_{k+n}(n)
& = &
\binom{k+n}{n}\cdot (-1)^{k} \cdot \big(\Delta^{k}\frak{b}\big)_{n}.
\end{array} \]

\noindent In particular, the sequence $\frak{b}_{K}$ is completely
monotone.

The proof works again by induction on $k$. The case $k=0$ holds by
definition. For the induction step we use the cone property
$\drawdelete \dnib \omega_{k+1} =
\omega_{k}$. Via~\eqref{DrawdeleteBinEqn} it gives:
\[ \begin{array}{rcl}
\omega_{k}(i)
& = &
\frac{i+1}{k+1}\cdot\omega_{k+1}(i+1) + \frac{k+1-i}{k+1}\cdot\omega_{k+1}(i),
\end{array} \]

\noindent so that:
\[ \begin{array}{rcl}
\omega_{k+1}(i)
& = &
\frac{k+1}{k+1-i}\cdot\omega_{k}(i) - \frac{i+1}{k+1-i}\cdot\omega_{k+1}(i+1).
\end{array} \]

\noindent We use this equation in the first step below.
\[ \begin{array}{rcl}
\lefteqn{\omega_{(k+1)+n}(n)}
\\
& = &
\frac{k+1+n}{k+1}\cdot\omega_{k+n}(n) - \frac{n+1}{k+1}\cdot\omega_{k+1+n}(n+1)
\\
& \smash{\stackrel{\text{(IH)}}{=}} &
\frac{k+1+n}{k+1}\cdot\binom{k+n}{n}\cdot (-1)^{k} \cdot 
   \big(\Delta^{k}\frak{b}\big)_{n} -
   \frac{n+1}{k+1}\cdot\binom{k+1+n}{n+1}\cdot (-1)^{k} \cdot 

   \big(\Delta^{k}\frak{b}\big)_{n+1}
\\
& = &
\binom{k+1+n}{n}\cdot (-1)^{k+1} \cdot \Big(\big(\Delta^{k}\frak{b}\big)_{n+1}
   - \big(\Delta^{k}\frak{b}\big)_{n}\Big)
\\
& = &
\binom{k+1+n}{n}\cdot (-1)^{k+1} \cdot \big(\Delta^{k+1}\frak{b}\big)_{n}.
\end{array} \]

This second claim implies that there is a unique distribution
$\omega\in\Giry([0,1])$ whose $K$-th moments equals $\frak{b}_{K} =
\omega_{K}(K)$.  But then:
\[ \begin{array}{rcll}
\omega_{K}(k)
\hspace*{\arraycolsep}=\hspace*{\arraycolsep}
\omega_{(K-k)+k}(k)
& = &
\binom{K}{k}\cdot (-1)^{K-k} \cdot \big(\Delta^{K-k}\frak{b}\big)_{k}
  & \mbox{by Claim 2}
\\
& = &
\binom{K}{k}\cdot \displaystyle \int_{0}^{1} x^{k}\cdot (1-x)^{K-k} 
   \intd \omega(x) \quad
  & \mbox{by Claim 1}
\\
& = &
\big(\binomial[K]\dnib \omega\big)(k).
\end{array} \]


As an aside, it is perhaps instructive to look at how the distribution
$\omega\in\Giry([0,1])$ is built in the proof of Hausdorff's
criterion.  One way is as a limiting measure of the sequence of
finite discrete distributions on $[0,1]$,
\begin{equation}
\label{ApproxDstEqn}
\begin{array}{rcl}
\overline{\omega}_{K}
& \coloneqq &
\displaystyle\sum_{k=0}^{K}\omega_K(k)\bigket{\tfrac k K}.
\end{array}
\end{equation}

\noindent This is illustrated in the following table, for the examples of the binomial cone and the uniform cone.
The graphs show $\bar\omega_K$
(plotted as densities with respect to the counting measure). 
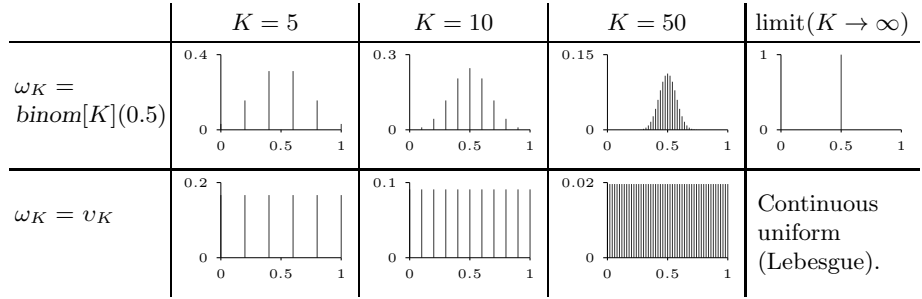
\newcommand{\discplot}[2]{%
  \begin{tikzpicture}[xscale=1.6, yscale=1]
        \begin{scope}[ycomb, yscale=1/#2]
            \draw[black!80, line width=0.1mm] plot #1;
          \end{scope}
          \draw[ultra thin] (0,0) -- (1,0);
          \draw[ultra thin](0,0) -- (0,1 );
     \foreach \x in {0,0.5,1}
        \draw[thin] (\x,-0pt) -- (\x,-1pt)
            node[anchor=north] {\tiny $\x$};
        \draw[thin] (-0pt,0) -- (-1pt,0)
            node[anchor=east] {\tiny $0$};
        \draw[thin] (-0pt,1) -- (-1pt,1)
            node[anchor=east] {\tiny $#2$};
    \end{tikzpicture}%
  }

\def\binomFiveDATA{(0,0.03125)(0.2,0.15625)(0.4,0.3125)(0.6,0.3125)(0.8,0.15625)(1,0.03125)}  \def\unifFiveDATA{(0,0.166666666666667)(0.2,0.166666666666667)(0.4,0.166666666666667)(0.6,0.166666666666667)(0.8,0.166666666666667)(1,0.166666666666667)} \def\binomTenDATA{(0,0.0009765625)(0.1,0.009765625)(0.2, 0.0439453125)(0.3,0.1171875)(0.4,0.205078125)(0.5,0.24609375)(0.6,0.205078125)(0.7,0.1171875)(0.8, 0.0439453125)(0.9,0.009765625)(1,0.0009765625)}
\def\unifTenDATA{(0,0.0909090909090909)(0.1,0.0909090909090909)(0.2,0.0909090909090909)(0.3,0.0909090909090909)(0.4,0.0909090909090909)(0.5,0.0909090909090909)(0.6,0.0909090909090909)(0.7,0.0909090909090909)(0.8,0.0909090909090909)(0.9,0.0909090909090909)(1,0.0909090909090909)}
\def\binomFiftyDATA{(0.26,0.000107824550266145)(0.28,0.000315179454624115)(0.3,0.00199913825504439)(0.32,0.00437311493290959)(0.34,0.00874622986581919)(0.36,0.0160347547540018)(0.38,0.027005902743582)(0.4,0.0418591492525522)(0.42,0.0597987846465031)(0.44,0.0788256706703905)(0.46,0.0959616860335188)(0.48,0.107956896787709)(0.5,0.112275172659217)(0.52,0.107956896787709)(0.54,0.0959616860335188)(0.56,0.0788256706703905)(0.58,0.0597987846465031)(0.6,0.0418591492525522)(0.62,0.027005902743582)(0.64,0.0160347547540018)(0.66,0.00874622986581919)(0.68,0.00437311493290959)(0.7,0.00199913825504439)(0.72,0.00083297427293516)(0.74,0.000315179454624115)(0.76,0.000107824550266145)}
\def\binomHundredDATA{(0.32,0.000112816971272231)(0.33,0.000232471334742779)(0.34,0.00045810527728724)(0.35,0.000863855665741653)(0.36,0.00155973939647799)(0.37,0.00269792760471867)(0.38,0.00447287997624411)(0.39,0.00711073226992655)(0.4,0.010843866711638)(0.41,0.0158690732365434)(0.42,0.0222922695465729)(0.43,0.0300686426442145)(0.44,0.0389525597890962)(0.45,0.0484742966264308)(0.46,0.0579583981402977)(0.47,0.0665904999909802)(0.48,0.0735270104067074)(0.49,0.0780286641050773)(0.5,0.0795892373871788)(0.51,0.0780286641050773)(0.52,0.0735270104067074)(0.53,0.0665904999909802)(0.54,0.0579583981402976)(0.55,0.0484742966264308)(0.56,0.0389525597890962)(0.57,0.0300686426442145)(0.58,0.0222922695465729)(0.59,0.0158690732365434)(0.6,0.010843866711638)(0.61,0.00711073226992655)(0.62,0.00447287997624411)(0.63,0.00269792760471867)(0.64,0.00155973939647798)(0.65,0.000863855665741653)(0.66,0.00045810527728724)(0.67,0.000232471334742779)(0.68,0.000112816971272231)}
\def\unifFiftyDATA{(0,0.0196078431372549)(0.02,0.0196078431372549)(0.04,0.0196078431372549)(0.06,0.0196078431372549)(0.08,0.0196078431372549)(0.1,0.0196078431372549)(0.12,0.0196078431372549)(0.14,0.0196078431372549)(0.16,0.0196078431372549)(0.18,0.0196078431372549)(0.2,0.0196078431372549)(0.22,0.0196078431372549)(0.24,0.0196078431372549)(0.26,0.0196078431372549)(0.28,0.0196078431372549)(0.3,0.0196078431372549)(0.32,0.0196078431372549)(0.34,0.0196078431372549)(0.36,0.0196078431372549)(0.38,0.0196078431372549)(0.4,0.0196078431372549)(0.42,0.0196078431372549)(0.44,0.0196078431372549)(0.46,0.0196078431372549)(0.48,0.0196078431372549)(0.5,0.0196078431372549)(0.52,0.0196078431372549)(0.54,0.0196078431372549)(0.56,0.0196078431372549)(0.58,0.0196078431372549)(0.6,0.0196078431372549)(0.62,0.0196078431372549)(0.64,0.0196078431372549)(0.66,0.0196078431372549)(0.68,0.0196078431372549)(0.7,0.0196078431372549)(0.72,0.0196078431372549)(0.74,0.0196078431372549)(0.76,0.0196078431372549)(0.78,0.0196078431372549)(0.8,0.0196078431372549)(0.82,0.0196078431372549)(0.84,0.0196078431372549)(0.86,0.0196078431372549)(0.88,0.0196078431372549)(0.9,0.0196078431372549)(0.92,0.0196078431372549)(0.940000000000001,0.0196078431372549)(0.960000000000001,0.0196078431372549)(0.980000000000001,0.0196078431372549)(1,0.0196078431372549)}
\def\unifHundredDATA{((0,0.0099009900990099)(0.01,0.0099009900990099)(0.02,0.0099009900990099)(0.03,0.0099009900990099)(0.04,0.0099009900990099)(0.05,0.0099009900990099)(0.06,0.0099009900990099)(0.07,0.0099009900990099)(0.08,0.0099009900990099)(0.09,0.0099009900990099)(0.1,0.0099009900990099)(0.11,0.0099009900990099)(0.12,0.0099009900990099)(0.13,0.0099009900990099)(0.14,0.0099009900990099)(0.15,0.0099009900990099)(0.16,0.0099009900990099)(0.17,0.0099009900990099)(0.18,0.0099009900990099)(0.19,0.0099009900990099)(0.2,0.0099009900990099)(0.21,0.0099009900990099)(0.22,0.0099009900990099)(0.23,0.0099009900990099)(0.24,0.0099009900990099)(0.25,0.0099009900990099)(0.26,0.0099009900990099)(0.27,0.0099009900990099)(0.28,0.0099009900990099)(0.29,0.0099009900990099)(0.3,0.0099009900990099)(0.31,0.0099009900990099)(0.32,0.0099009900990099)(0.33,0.0099009900990099)(0.34,0.0099009900990099)(0.35,0.0099009900990099)(0.36,0.0099009900990099)(0.37,0.0099009900990099)(0.38,0.0099009900990099)(0.39,0.0099009900990099)(0.4,0.0099009900990099)(0.41,0.0099009900990099)(0.42,0.0099009900990099)(0.43,0.0099009900990099)(0.44,0.0099009900990099)(0.45,0.0099009900990099)(0.46,0.0099009900990099)(0.47,0.0099009900990099)(0.48,0.0099009900990099)(0.49,0.0099009900990099)(0.5,0.0099009900990099)(0.51,0.0099009900990099)(0.52,0.0099009900990099)(0.53,0.0099009900990099)(0.54,0.0099009900990099)(0.55,0.0099009900990099)(0.56,0.0099009900990099)(0.57,0.0099009900990099)(0.58,0.0099009900990099)(0.59,0.0099009900990099)(0.6,0.0099009900990099)(0.61,0.0099009900990099)(0.62,0.0099009900990099)(0.63,0.0099009900990099)(0.64,0.0099009900990099)(0.65,0.0099009900990099)(0.66,0.0099009900990099)(0.67,0.0099009900990099)(0.68,0.0099009900990099)(0.69,0.0099009900990099)(0.7,0.0099009900990099)(0.71,0.0099009900990099)(0.72,0.0099009900990099)(0.73,0.0099009900990099)(0.74,0.0099009900990099)(0.75,0.0099009900990099)(0.76,0.0099009900990099)(0.77,0.0099009900990099)(0.78,0.0099009900990099)(0.79,0.0099009900990099)(0.8,0.0099009900990099)(0.81,0.0099009900990099)(0.820000000000001,0.0099009900990099)(0.830000000000001,0.0099009900990099)(0.840000000000001,0.0099009900990099)(0.850000000000001,0.0099009900990099)(0.860000000000001,0.0099009900990099)(0.870000000000001,0.0099009900990099)(0.880000000000001,0.0099009900990099)(0.890000000000001,0.0099009900990099)(0.900000000000001,0.0099009900990099)(0.910000000000001,0.0099009900990099)(0.920000000000001,0.0099009900990099)(0.930000000000001,0.0099009900990099)(0.940000000000001,0.0099009900990099)(0.950000000000001,0.0099009900990099)(0.960000000000001,0.0099009900990099)(0.970000000000001,0.0099009900990099)(0.980000000000001,0.0099009900990099)(0.990000000000001,0.0099009900990099)(1,0.0099009900990099)}
  \[{\small\begin{array}{p{2cm}|c|c|c|c}
    &K=5&K=10&K=50& \text{limit}(K\to \infty)
    \\\hline
    \vspace{-1cm}${\omega_K=}$ $\binomial[K](0.5)$&\discplot{coordinates \binomFiveDATA}{0.4}&
\discplot{coordinates \binomTenDATA}{0.3}
                 &\discplot{coordinates \binomFiftyDATA}{0.15}&\discplot{coordinates {(0.5,1)}}{1}\\\hline
    \vspace{-1cm}${\omega_K=}$ $ \upsilon_K$&\discplot{coordinates \unifFiveDATA}{0.2}&\discplot{coordinates \unifTenDATA}{0.1}&\discplot{coordinates \unifFiftyDATA}{0.02}&\parbox{2cm}{\vspace{-1.2cm}Continuous uniform (Lebesgue).}
                          \end{array}}\]


We conclude by explaining why this construction gives a measurable
mediating map.  Suppose we have a measurable space $(Y,\Sigma)$ and a
sequence of channels $c_K\colon Y\chanto \{0,\ldots,K\}$ forming a
cone over~\eqref{BinomChainDiag}, \ie~$c_K=\drawdelete
\klafter c_{K+1}$ for all $K\in \NNO$.  We define a mediating channel
$c\colon Y\chanto [0,1]$ by putting the distribution~$c(y)$ as the
unique one with moments $c_K(y)(K)$, as above.  It remains to show
that $c$, regarded as a function $Y\to \Giry([0,1])$, is a measurable
function.  It is immediate that the sequence-of-moments function
$\frak{m} \colon Y\to [0,1]^\infty$ given by:
\[ \begin{array}{rcccl}
\frak{m}(y)_K
& \coloneqq &
\displaystyle \int_0^1 r^K\,c(x)(\dd r)
& = &
c_K(y)(K).
\end{array} \]

\noindent is measurable, because the $c_K$'s are
assumed to be channels: let $V \subseteq [0,1]$ be measurable; for
associated basic measurable subsets of $[0,1]^{\infty}$ one has:
\[ \begin{array}{rcccl}
\frak{m}^{-1}\big([0,1]^{K} \times V \times [0,1]^{\infty}\big)
& = &
\setin{y}{Y}{c_{K}(y)(K) \in V}
& = &
c_{K}^{-1}\big([0,1]^{K} \times V\big).
\end{array} \]

\noindent Further, $c\colon Y\to \Giry ([0,1])$ factors through
$\frak{m}$; indeed $\Giry([0,1])$ can be thought of as a subset of
$[0,1]^\infty$, when each distribution is regarded as its sequence of
moments.  We conclude by the less well-known fact that $\Giry([0,1])$
is a sub\emph{space} of $[0,1]^\infty$.
\qed

\auxproof{
The following is Bart's original argument.
The probability density of
this $\omega$ is obtained as limit of step-functions. The four main
parts of the construction are illustrated in
Figure~\ref{FinettiChartsFig}.

We turn each $\omega_K$ into a step-function $\overline{\omega}_K$ on
the unit interval $[0,1]$. For a fixed $K$ we define a `splitting' function
$s_{K} \colon [0,1] \rightarrow \{0,1,\ldots,K\}$ by:
\[ \begin{array}{rclcl}
s_{K}(x)
& \coloneqq &
k
& \quad\Longleftrightarrow\quad &
\frac{k}{K+1} \leq x < \frac{k+1}{K+1}.
\end{array} \]

\noindent We need to add as border case $s_{K}(1) = K$. Via this
function $s_{K}$ we split up $[0,1]$ into $K+1$ intervals
$[0,\frac{1}{K+1})$, $[\frac{1}{K+1},\frac{2}{K+1})$, \ldots
    $[\frac{K}{K+1}, 1)$, via inverse images $s_{K}^{-1}(k) =
      [\frac{k}{K+1}, \frac{k+1}{K+1})$. We now define a function
        $\overline{\omega_K} \colon [0,1] \rightarrow \nnR$ via:
\[ \begin{array}{rcl}
\overline{\omega}_{K}(x)
& \coloneqq &
(K+1)\cdot \omega_{K}\big(s_{K}(x)\big).
\end{array} \]

\noindent This $\overline{\omega}_K$ forms a probability density
function (pdf) on $[0,1]$, since:
\[ \begin{array}{rcccl}
\displaystyle\int_{0}^{1} \overline{\omega}_K(x) \intd x
& = &
\displaystyle\sum_{0\leq k \leq K} \textstyle\frac{1}{K+1}\cdot (K+1)\cdot
   \omega_{K}(k)
& = &
1.
\end{array} \]


\marginpar{SS: Are you sure this is ok? What about $\omega_K=\binomial[K](0.5)$? Then we'd want $\omega=\delta_{0.5}$, which has no density? Shouldn't it be convergence in distribution?}
By the Helly theorem~\cite{??} there is a (unique) distribution
$\omega \in\Giry([0,1])$ given as $\omega(M) = \int_{M} f(x)\intd x$
for a pdf $f$ on $[0,1]$ that arises as limit of the
$\overline{\omega}_{K}$. Then, for each $N$,
\[ \begin{array}[b]{rcl}
\omega_{K}(k)
& = &
\big(\drawdelete^{N} \dnib \omega_{K+N}\big)(k)
\\
& = &
\displaystyle\sum_{\ell\leq K+N} \omega_{K+N}(\ell) \cdot \drawdelete^{N}(\ell)(k)
\\[+1em]
& = &
\displaystyle\sum_{\ell\leq K+N} \textstyle \frac{1}{K+N-1} \cdot
   \overline{\omega}_{K+N-1}(\nicefrac{\ell}{K+N}) \cdot \drawdelete^{N}(\ell)(k)
\\[+1em]
& \rightarrow &
\displaystyle\int_{0}^{1} f(x) \cdot \binomial[K](x)(k) \intd x,
   \qquad \mbox{as $N \rightarrow \infty$}
\\
& = &
\big(\binomial[K] \dnib \omega\big)(k).
\end{array} \eqno{\QEDbox} \]



}
\end{proof}

Theorem~\ref{FinettiThm} is a reformulation of a classical result of
de Finetti~\cite{Finetti30}. It introduced the concept of
\emph{exchangeability} of random variables which has disappeared
completely from our reformulation, but is implicit in our use of
multisets. This requires some explanation.

Fix a measurable space $A$ and a distribution $\rho\in\Giry(A)$. A
\emph{random variable} is a function $U\colon A \rightarrow
\R$. Usually, the notation $P(U)$ is used for the image distribution
$\Giry(U)(\rho)$ on $\R$. One often also writes $P(U=k)$ for the
probability $\Giry(U)(\rho)(\{k\})$, for $k\in\R$, and so on.

Let $U_{1},\dots,U_{n}$ be a finite
sequence of random variables. It induces a joint distribution on $\R^n$,
\[ \begin{array}{rcl}
P(U_{1}, \ldots, U_{n})
& \coloneqq &
\Giry\big(U_{1} \times \cdots\times U_{n}\big)
   \big(\rho \otimes \cdots \otimes \rho\big),
\end{array} \]

\noindent where $U_{1}\times\cdots\times U_{n} \colon A^n\rightarrow \R^{n}$ and $(\rho \otimes \cdots
\otimes \rho) \in \Giry(A^n)$ the product measure. 

Such a sequence $(U_{i})$ is called \emph{exchangeable}
if for each permutation $\pi$:
\[ \begin{array}{rcl}
P(U_{1}, \ldots, U_{n})
& = &
P(U_{\pi(1)}, \ldots, U_{\pi(n)}).
\end{array} \]

\noindent An infinite sequence is called exchangeable if each finite
subsequence is exchangeable.

The original form of de Finetti's theorem involves binary (0-1) random
variables with $U_{i} \colon A_{i} \rightarrow 2 = \{0,1\}
\hookrightarrow \R$. The distribution $P(U_{i})$ is then on a
`Bernouilli' distribution on $2$. When random variables are
exchangeable their order doesn't matter but only how many of them are
zero or one. This is captured by their sum $\sum_{i}U_{i}$. But this
sum is a \emph{multiset} over $2$. Indeed, in a multset, the order of
elements does not matter, only their multiplicities (numbers of
occurrences). This is a crucial observation underlying our
reformulation, that already applied to the accumulator map
$\mathsl{acc}$ in~\eqref{AccDiag}. But let's first look at the usual
formulation of de Finetti's (binary) representation theorem,
see~\cite{Finetti30} and the textbooks~\cite{Feller70,Klenke13}.

\begin{theorem}
\label{OriginalFinettiThm}
Let $U_{1}, U_{2}, \ldots$ be an infinite sequence of exchangeable
binary random variables. Then there is a probability measure $\omega$
on $[0,1]$ such that for all~$k\leq K$,
\[ \begin{array}{rcl}
P(\sum_{i=1}^KU_{i}=k)
& = &
\displaystyle\int_{0}^{1} r^{k}(1-r)^{(K-k)}\,\omega(\intd r).
\end{array} \]
\end{theorem}

It is clear that there is a considerable gap between this original
formulation and the inverse limit formulation that has been developed
in this paper. There is room for narrowing this gap.

The essence that is used in de Finetti's theorem about $n$
exchangeable binary random variables are their outcomes via the
accumulator map $\mathsl{acc}$ from~\eqref{AccDiag}, giving a map
$\Dst(\mathsl{acc}) \colon \Dst(2^{n}) \rightarrow \Dst(\Mlt[n](2))$.


In our approach we have reformulated the requirements about random variables and
about exchangeability and worked directly with distributions on
multisets, see Proposition~\ref{BinomChainProp} and
Theorem~\ref{FinettiThm}. De Finetti's
(binary) representation Theorem~\ref{OriginalFinettiThm} becomes a limit
statement in a Kleisli category. We think that this approach brings
forward the mathematical essentials.

\section{A final exchangeable coalgebra}\label{ExchangeSec}

We return to the setting of coalgebras, begun in
Section~\ref{PolyaSec}.  Any coalgebra $h\colon X\chanto 2\times X$
for the functor $2\times (-)$ on $\Kl(\Dst)$ or on $\Kl(\Giry)$ is a
probabilistic process that produces a random infinite sequence of
boolean values.  Following~\cite{KerstanK12}, we can understand this
induced sequence as arising from the fact that the pair of head and
tail function
\[ \xymatrix{
  2^\infty\ar[r]^-{\cong} & 2\times 2^\infty
} \]

\noindent is a final coalgebra in the category $\Kl(\Giry)$, and so there is a unique channel
$h^\sharp \colon X\chanto 2^\infty$ assigning to each element of $X$
the random sequence that it generates. Here $2^\infty$ is the
measurable space of infinite binary streams, regarded with the product
$\sigma$-algebra.

It is helpful to spell this construction out a little.
The sequence of first projections
\begin{equation}
\label{eqn:kolmogorov}
\vcenter{\xymatrix{
2^0 & 2^1\ar[l] & 2^2\ar[l] & 2^3\ar[l] & \cdots\ar[l]
}}
\end{equation}

\noindent has as a limit the space $2^\infty$, with the limiting cone
being the obvious projections.  And this limit is also a limit in
$\Kl(\Giry)$. (This is a categorical formulation of Kolmogorov's
extension theorem, see~\cite[Thm.~2.5]{DanosG15}.)

Now, as usual in coalgebra theory, we find $h^\sharp\colon X\chanto
2^\infty$ for any coalgebra $h\colon X\chanto 2\times X$ by noticing
that the channels $h^\sharp_K\colon X\chanto 2^K$
\eqref{eqn:channelsequence} define a cone over the chain
\eqref{eqn:kolmogorov}.  The limiting channel $h^\sharp\colon X\chanto
2^\infty$ for the cone is the unique coalgebra homomorphism to the
final coalgebra.

\begin{example}
\label{CoinCoalgEx}
We define the \emph{bernoulli coalgebra} to be the channel $\coin
\colon [0,1]\chanto 2\times [0,1]$ given by $\coin(r) \coloneqq r\cdot
\ket{1,r} + (1-r)\cdot \ket{0,r}$.  In words, $\coin(r)$ makes a coin
toss with known bias~$r$ and returns the outcome together with the same
bias~$r$ unchanged.

The final channel $\iid \coloneqq \coin^{\sharp} \colon[0,1]\chanto
2^\infty$ takes a probability $r$ to a random infinite sequence of
independent and identically distributed binary values, all distributed
as $r\ket 1+(1-r)\ket 0$. More explicitly, for basic measurable
subsets, at each stage $n$,
\[ \begin{array}{rclcrcl}
\iid(r)\big(2^{n} \times \{1\} \times 2^{\infty}\big)
& = &
r
\quad\mbox{and}\quad
\iid(r)\big(2^{n} \times \{0\} \times 2^{\infty}\big)
& = &
1-r.
\end{array} \]
\end{example}

So, any coalgebra defines a cone over~\eqref{eqn:kolmogorov}.  We have
shown in Lemma~\ref{ExchCoalgLem} that any \emph{exchangeable}
coalgebra defines a cone over the draw-delete chain~\eqref{eqn:chain}.
We use this to show that the bernoulli coalgebra $\coin\colon
[0,1]\chanto 2\times [0,1]$ is the final exchangeable coalgebra.

The final coalgebra $2^\infty\chanto 2\times 2^\infty$ is not
exchangeable.  To see this, consider the infinite alternating sequence
$01010101\dots\in 2^\infty$, so that the left hand side of
\eqref{eqn:exch} at the sequence gives $(0,1,010101\dots)$, whereas
the right hand side of \eqref{eqn:exch} at the sequence gives
$(1,0,010101\dots)$, which is different.

The bernoulli coalgebra $\coin \colon [0,1]\chanto 2\times [0,1]$
\emph{is} exchangeable, because
\[ \begin{array}{rcl}
(\idmap\otimes \coin)\dnib \coin(r)
& = &
r^2\bigket{1,1,r} \;+
\\
& & 
(1-r)r\bigket{0,1,r} + r(1-r)\ket{1,0,r} \; +
\\
& & 
(1-r)^2\bigket{0,0,r}\text.
\end{array} \]

\noindent Notice that reordering the outputs does not change the statistics.

\begin{theorem}
\label{FinalThm}
The bernoulli coalgebra $\coin \colon [0,1]\chanto 2\times [0,1]$ is
the final exchangeable coalgebra.
\end{theorem}

\begin{proof}
Let $h\colon X \chanto 2\times X$ be an exchangeable coalgebra. By
Lemma~\ref{ExchCoalgLem} the associated maps $h_{K} \colon X \chanto
\Mlt[K](2)$ form a cone. By de Finetti's theorem in limit form ---
that is, by Theorem~\ref{FinettiThm} --- we get a unique mediating
channel $f\colon X \chanto [0,1]$ with $\binomial[K] \klafter f =
h_{K}$ for each $K$. Our aim is to show that $f$ is a unique coalgebra
map: $\coin \klafter f = (\idmap\otimes f) \klafter h$.

We do know that the limiting map $\smash{\xymatrix{X\ar[r]|-\circ ^-f
    & [0,1]\ar[r]|-\circ^-\iid&2^\infty}}$ is a unique coalgebra
homomorphism.  So to show that $f$ is a unique coalgebra homomorphism,
it suffices to show that $\iid$ is a monomorphism in $\Kl(\Giry)$.

To see this, first notice that the $K$-th moment of some distribution
$\omega\in\Giry([0,1])$ is the probability
$(\iid\dnib\omega)(\{1\}^K\times 2^\infty)$, and so if two random
sequences are equal, $\iid \dnib \omega_1=\iid\dnib\omega_2$, then in
particular the moments of $\omega_1$ and $\omega_2$ are equal and so
$\omega_1=\omega_2$.  Moreover the endofunctor $2\times-$ preserves
monomorphisms.  So if the final channel $X\chanto 2^\infty$ factors
through $\iid:[0,1]\chanto 2^\infty$ as a channel then it also factors
through as a coalgebra homomorphism, necessarily uniquely, by the
definition of monomorphism. \QED
\end{proof}

  \begin{example}
    We return to the example of P\'olya's urn coalgebra
    $\mathsl{pol}\colon \neMlt(2)\chanto 2\times \neMlt(2)$. 
    We have shown in Proposition~\ref{PolyaIteratedProp} that the cone $\mathsl{pol}_K:\neMlt(2)\to \Mlt[K](2)$ defined by iterated draws 
    is given by the beta-bernoulli distributions: $\mathsl{pol}_K(\varphi)=\mathsl{betabern}[K](\varphi(0),\varphi(1))$,
    and we have shown in Proposition~\ref{BetaBinomLem} that the 
    beta-bernoulli cone over~\eqref{eqn:chain} 
    has mediating map the beta distribution.
    So the final channel
    $\textit{beta}:\neMlt(2)\chanto [0,1]$
    is given by putting $\textit{beta}(\varphi)$ as the beta distribution with parameters
    $(\varphi(0),\varphi(1))$,
    and this is a coalgebra homomorphism.
    The commuting diagram
    \[\xymatrix{
        \ar[d]|-\circ_{\textit{beta}}\neMlt(2)\ar[r]|-\circ ^-{\textsl{pol}}&2\times \neMlt(2)\ar[d]|-\circ^{2\times \textit{beta}}
        \\
        [0,1]\ar[r]|-\circ_-\coin &2\times[0,1]}
    \]
    expresses the `conjugacy' between the bernoulli and beta distributions,
see also~\cite{Jacobs20a}.
\end{example}

\section{Conclusions and further work}

This paper has translated the classical representation theorem of de
Finetti into the language of multisets and distributions and turned it
into a limit theorem. This limit is subsequently used for a finality
result for exchangeable coalgebras.

We have concentrated, as is usual for the de Finetti theorem, on the
binary case, giving a limit statement for the special case $X = 2 =
\{0,1\}$ of the cone in Proposition~\ref{BinomChainProp}. There is an
obvious question whether this formulation can be extended to $X =
\{0,1,\ldots,n-1\}$ for $n > 2$, or to more general measurable
spaces. We see two possible avenues:
\begin{itemize}
\item extension to finite sets via a multivariate version of
Hausdorff's moments theorem, see~\cite{KleiberS13};

\item extension to Polish spaces, following~\cite{DahlqvistDG16} and
  its categorical justification for extending de Finetti's theorem to
  such spaces (esp.\ Theorems~35 and~40) in terms of exchangeability.
\end{itemize}

\noindent Another point for future work is to investigate whether the
categorical formulation of de Finetti's theorem could be taken as an
axiom for a `synthetic probability theory' (\eg~following
\cite{Fritz20,FritzR19,ScibiorKVSYCOMHG18}), and indeed whether it
holds in some proposed categorical models which do support
de-Finetti-like arguments (\cite{StatonKSY17,StatonSYAFR15}).

Finally, we note that chains of channels between multisets also arise
in recent exponential modalities for linear
logic~\cite{DahlqvistK20,hamano-exponential}; the connections remain
to be worked out.

\paragraph{Acknowledgements.}
Staton is supported by a Royal Society Fellowship, and has enjoyed
discussions about formulations of de Finetti's theorem with many
people including coauthors on~\cite{StatonSYAFR15}.

\bibliography{bib}

\end{document}